\documentclass[final,1p,times]{elsarticle}
\usepackage{amssymb}
\usepackage{amsmath}
\usepackage{hyperref}
\usepackage{color}
\usepackage{xcolor}
\usepackage{euscript}

\newcommand{\norm}[2]{\|#1\|_{#2}}
\newcommand{\scal}[2]{{\left\langle{{#1},{#2}}\right\rangle}}
\newcommand{\qbox}[1]{\quad\hbox{#1}\quad}
\newcommand{\qqbox}[1]{\qquad\hbox{#1}\qquad}

\newcommand{\R}{\ensuremath{{\mathbb R}}}

\newcommand{\prox}{\ensuremath{\operatorname{prox}}}
\newcommand{\J}{\ensuremath{\left[0,+\infty\right)}}

\newcommand{\zer}
{\mathrm{zer}}

\newcommand{\off}[1]{}

\newcommand{\ps}[2]{\langle #1,#2\rangle}

\newtheorem{theorem}{Theorem}[section]
\newtheorem{lemma}[theorem]{Lemma}

\newtheorem{proposition}[theorem]{Proposition}
\newtheorem{assumption}[theorem]{Assumption}

\newtheorem{example}[theorem]{Example}
\newtheorem{remark}[theorem]{Remark}

\journal{Hybrid systems}

\begin{document}

\begin{frontmatter}

\title{Preconditioned primal-dual dynamics in convex optimization: non-ergodic convergence rates}

\author[1]{Vassilis Apidopoulos}
\ead{vassilis.apid@gmail.com}

\author[2]{Cesare Molinari}
\ead{cecio.molinari@gmail.com}

\author[3]{Juan Peypouquet\corref{cor1}}
\ead{j.g.peypouquet@rug.nl}

\author[2]{Silvia Villa}
\ead{silvia.villa@unige.it}

\cortext[cor1]{Corresponding author}

\affiliation[1]{organization={Archimedes, Athena Research Center},
	addressline={1 Artemidos street}, 
	city={Athens},
	%               citysep={}, % Uncomment if no comma needed 
	%                           % between city and postcode
	postcode={15125},
	country={Greece}}

\affiliation[2]{organization={Machine Learning Genoa Center (MaLGa) – Dipartimento di Matematica – Università degli Studi di Genova},
	addressline={Via Dodecaneso 35}, 
	postcode={16146}, 
	postcodesep={}, 
	city={Genoa},
	country={Italy}}

\affiliation[3]{organization={University of Groningen, Faculty of Science and Engineering, 	Bernoulli Institute},
	addressline={Nijenborgh 9}, 
	city={Groningen},
	postcode={9747AG}, 
	state={Groningen}, 
	country={The Netherlands}}

\begin{abstract}
We introduce and analyze a continuous primal-dual dynamical system in the context of the  
minimization problem $f(x)+g(Ax)$, where $f$ and $g$ are convex functions and $A$ is a linear operator.
In this setting, the trajectories of the Arrow-Hurwicz continuous flow may not converge, accumulating at points that are not solutions. 
Our proposal is inspired by the primal-dual algorithm \cite{ChaPoc11}, where convergence and splitting on the primal-dual variable are ensured by adequately preconditioning the proximal-point algorithm. We
consider a family of preconditioners, which are allowed to depend on time and on the operator $A$, but not on the functions $f$ and $g$, and analyze asymptotic properties of the corresponding preconditioned flow. Fast convergence rates for the primal-dual gap and optimality of its (weak) limit points are obtained, in the general case, for asymptotically antisymmetric preconditioners, and, in the case of linearly constrained optimization problems, under milder hypotheses. Numerical examples support our theoretical findings, especially in favor of the antisymmetric preconditioners.
\end{abstract}

\begin{keyword}
Convex optimization \sep Saddle-point problem \sep Primal-Dual methods

\MSC[2020] 34D05 \sep 65K05 \sep 65K10 \sep 90C25 

\end{keyword}

\end{frontmatter}

\section{Introduction}\label{sec:sbnsm}

Given two Hilbert spaces $\mathcal{X}$ and $\mathcal{Y}$, along with functions $f\in\Gamma_0(\mathcal{X})$ and $g\in \Gamma_0(\mathcal{Y})$, and a bounded linear operator $A:\mathcal X\to\mathcal Y$, consider the following optimization problem: 
\begin{equation}\label{opt_probl}
    \min_{x\in \mathcal{X}} \ f(x)+g(Ax).
\end{equation}
This structure often arises in variational analysis
and PDEs \cite{AttButMic06}, machine learning \cite{BacJenMai11}, inverse problems from imaging and signal theory \cite{ChaPoc16}, optimal control \cite{DonStaBoy11}, game theory \cite{YiPav19}. 

Primal-dual splitting algorithms decouple the problem components without stringent regularity assumptions. A notable example is the primal-dual algorithm proposed in \cite{ChaPoc11}, which iterates
\begin{equation}\label{pd_algo}
	\begin{cases}
 x_{k+1}=\prox_{\sigma f} \left(x_k-\sigma A^*y_k\right)\\
 y_{k+1}= \prox_{\tau g^*}\left(y_k + \tau A(2x_{k+1}-x_k)\right).
	\end{cases}
\end{equation}
This algorithm has been generalized to cope with more general settings, and its convergence properties are well understood (see also \cite{Vu13,Con13}). Algorithm~\eqref{pd_algo} can be equivalently written as 
\begin{equation} \label{eq:discr_pre}
		\underbrace{\begin{pmatrix}
			I && -\tau A^*\\[1.3ex]
			-\sigma A && I 
		\end{pmatrix}}_{=P}
  \begin{pmatrix}
			\dfrac{x_{k+1}-x_k}{\sigma}\\
			\dfrac{y_{k+1}- y_k}{\tau}
		\end{pmatrix}
		+\underbrace{\begin{pmatrix}
			\partial f && A^*\\
			-A && \partial g^* 
		\end{pmatrix}}_{=M}
  \begin{pmatrix}
			x_{k+1}\\
			y_{k+1}
		\end{pmatrix}
  \ni \begin{pmatrix}
			0\\
		  0 
		\end{pmatrix},
\end{equation}
where we identify a linear {\it preconditioner} $P$, and a maximally monotone operator $M$, which is the sum of a {\it diagonal} and {\it demi-positive} maximally monotone operator, on one hand, and a {\it skew-symmetric} bounded linear operator, on the other. Therefore, \eqref{eq:discr_pre} is a {\it preconditioned} proximal-point algorithm. 

\subsection*{A continuous-time model}

Letting $\sigma,\tau\to 0$, and writing $z=(x,y)^T$, an underlying continuous-time process is revealed, in line with the classical Arrow-Hurwicz dynamical system \cite{duguid1960studies, arrow2014gradient}, namely
\begin{equation}\label{arrow}
    \dot{z}(t)+Mz(t)\ni 0.
\end{equation}
The continuous-time point of view provides new insights and deeper understanding of the discrete optimization methods, thanks to certain advantages of the continuous time formulation. Without further assumptions on the functions $f$ and $g$, the solutions to the differential inclusion \eqref{arrow} converge only ergodically to primal-dual solutions of problem \eqref{opt_probl}, which are the zeroes of $M$ (see, for instance, \cite{brezis1973ope,PeySor10}). A general min-max problem has been studied under this perspective in \cite{holding2014convergence}, while a projected variant is analyzed in \cite{cherukuri2017saddle, cherukuri2017role}. Stability issues are studied in \cite{holding2020stability, holding2020stabilityt}. The Arrow-Hurwicz dynamics for a (strictly) convex mathematical programming problem is investigated in \cite{cherukuri2016asymptotic,feijer2010stability}.
If $g$ is the indicator function of a point $b\in\mathcal{Y}$, \eqref{opt_probl} becomes the linearly constrained problem
\begin{equation}\label{lin_prob}
\min_{x} f(x) \text{ s.t. } Ax=b.  
\end{equation}
In this setting, although the trajectories generated by \eqref{arrow} may not converge\footnote{Take, for instance, $\mathcal Y=\mathcal X=\R$, $f\equiv 0$, $g^*\equiv 0$ and $A=I$, which reduces to rotations around the origin in $\R^2$.}, convergence rates on the ergodic primal-dual gap have been obtained in \cite{niederlander2024arrow}. In the search for non-ergodic convergence, a Tikhonov regularization approach is developed in \cite{battahi2024asymptotic}, where they obtain strong convergence of the trajectory to the least-norm element of the primal-dual solution set, along with convergence rates for the primal-dual gap. An alternative perspective, based on the augmented Lagrangian is explored in \cite{ozaslan2024stability}.
A continuous counterpart of ADMM is analyzed in \cite{LiShi24}, while a dual subgradient flow is investigated in \cite{apidopoulos2023regularization}.
Extensions of the above-mentioned approaches to inertial systems have been considered in \cite{attouch2022fast, sun2024inertial, ding2024fast, battahi2025simultaneous}, also combined with Tikhonov regularization and Hessian-driven damping.

However, it is important to highlight that the preconditioned character of the primal-dual algorithm that we observe in \eqref{eq:discr_pre}---which plays a critical role in both theoretical analysis and practical performance---is lost in the limiting process leading to \eqref{arrow}. 
Broadly speaking, our aim is to restore a form of preconditioning in the continuous-time regime, in order to recover the non-ergodic behavior of \eqref{eq:discr_pre}.

\subsection*{A preconditioned Arrow-Hurwicz scheme}

A compromise between the preconditioned features of \eqref{eq:discr_pre}, characterized by the fixed step sizes, and the non-preconditioned continuous-time system \eqref{arrow} is achieved by allowing the preconditioner $P$ to evolve---and possibly degenerate!---over time. Motivated by the structure of $P$ in \eqref{eq:discr_pre}, we propose to set
$$
    P(t)=\begin{pmatrix}
			\alpha(t) I && \beta(t) A^*\\
			\gamma(t) A && \delta(t) I 
\end{pmatrix},
$$
where $\alpha, \beta, \gamma$ and $\delta$ are scalar functions to be specified later, and analyze the system defined by
\begin{equation} \label{prec_system}
\left\{\begin{array}{rcl}
\alpha(t)\dot{x}(t)+\beta(t) A^*\dot{y}(t) + \partial f(x(t))+A^*y(t) & \ni & 0\\
\gamma(t) A\dot{x}(t)+\delta(t) \dot{y}(t)-Ax(t)+\partial g^*(y(t)) & \ni & 0,
\end{array}\right.
\end{equation} 
for $t>0$, starting from a point $(x(0),y(0))=(x_0,y_0)\in \mathcal{X}\times \mathcal{Y}$. Our general goal is to understand the effect of these time-varying preconditioners on the qualitative and quantitative asymptotic properties of the Arrow-Hurwicz system.

The effects of a time-evolving preconditioner on these types of systems has been studied in \cite{luo2022primal} and \cite{li2024understanding}, although both consider constant non-diagonal coefficients. In \cite{luo2022primal}, the author studies a particular case of \eqref{prec_system} in the context of problem \eqref{lin_prob} for a differentiable function $f$. The proposed time-dependent preconditioner is {\it triangular}, with constant non-diagonal terms, namely $\beta\equiv 0$ and $\gamma\equiv -1$. In \cite{li2024understanding}, the authors study continuous and discrete dynamics for the preconditioned primal-dual flow system with $\beta=\gamma\equiv -s$, with $s$ being a positive constant. Our unified framework allows us to explore other combinations of parameters that guarantee non-ergodic convergence and rates, for problem \eqref{opt_probl}, and not only for the linearly constrained problem \eqref{lin_prob}.

\subsection*{Main contributions and organization of the paper}

The rest of the paper is organized as follows. In Section \ref{S:estimates}, we set the notation, and establish the energy estimates for the dynamical system \eqref{prec_system} that will be at the core of the subsequent analysis. In Section \ref{S:lin_constr}, the main estimation is used in order to provide a detailed study of the linearly constrained case \eqref{lin_prob}. By showcasing the relationship between the preconditioner and the rate at which the duality gap vanishes, we are able to establish a linear convergence rate for the latter (as in \cite{luo2022primal}), that can be made faster by proper tuning. Surprisingly enough, the rates are relatively oblivious to the non-diagonal coefficients. Constant diagonal coefficients give the slowest rates. We also conclude that every weak subsequential limit point of the trajectories is a primal-dual solution. In Section \ref{S:antisym}, we analyze the asymptotically antisymmetric case, where $\beta+\gamma\to 0$, in the context of problem \eqref{opt_probl}. Our main result establishes a non-ergodic convergence rate for the primal-dual gap, which can be sublinear, linear or faster, depending on the choice of the coefficients. Weak subsequential limit points of the trajectories are primal-dual solutions, as before. In the degenerate case of a diagonal preconditioner, the non-ergodic behavior is lost. In the symmetric case, where $\beta=\gamma$, the situation is similar to the non-preconditioned system, in the sense that both weak convergence and convergence rates are valid for the ergodic trajectory. This is not totally surprising, in view of its connection with \eqref{arrow}, once a time reparameterization has been carried out. Our new results, complementary to those in \cite{li2024understanding}, are presented in \ref{S:symmetric}, for the sake of completeness. Section \ref{S:num} contains some numerical experiments that give additional insight into the behavior observed in practice. Conclusions and some open questions are outlined in Section \ref{S:conclusions}.

\section{General estimates} \label{S:estimates}

A {\it primal-dual solution} of \eqref{opt_probl} is a point $(\bar{x},\bar{y})\in\mathcal{X} \times \mathcal{Y}$ satisfying 
\begin{equation}\label{sol}
	\begin{cases}
-A^*\bar y  \in  \partial f(\bar x)\\
A\bar x  \in \partial g^*(\bar y),
	\end{cases}
\end{equation}
or, equivalently,
\[
\left\{\begin{array}{rcl}
-A^*\bar y & \in & \partial f(\bar x)\\
\bar y & \in & \partial g(A\bar x).
\end{array}\right.
\]
If $(\bar{x},\bar{y})$ is a primal-dual solution, then $\bar x$ solves \eqref{opt_probl}. Indeed,
$$0=-A^*\bar y+A^*\bar y\in\partial f(\bar x)+A^*\partial g(A\bar x)\subset\partial(f+g\circ A)(\bar x).$$
The set of primal-dual solutions of \eqref{opt_probl} is denoted by $\cal S$. Being a primal-dual solution is equivalent to being a {\it saddle point} of the {\it Lagrangian} function $L:\mathcal{X} \times \mathcal{Y}\to\R$, defined by
$$ L(u,v):=f(u)+\langle v,Au\rangle-g^*(v),$$
which means that 
\begin{equation} \label{eq:saddle_points}
L(\bar x,v)\le L(\bar{x},\bar{y}) \le L(u,\bar y)
\end{equation}
for every $(u,v)\in\mathcal{X} \times \mathcal{Y}$. The {\it duality gap} along a trajectory $(x,y):[0,+\infty)\to \mathcal{X} \times \mathcal{Y}$ is
\begin{eqnarray}
    \Delta_{u,v}(t) & := & L\big(x(t),v\big)-L\big(u,y(t)\big)  \label{eq:duality_gap} \\
    & = & \left[f\big(x(t)\big)+\langle v,Ax(t)\rangle-g^*(v)\right]-\left[f(u)+\langle y(t),Au\rangle-g^*\big(y(t)\big)\right] \nonumber \\
    & = & \left[f\big(x(t)\big)-f(u)+\langle A^*v,x(t)-u\rangle\right]+\left[g^*\big(y(t)\big)-g^*(v)+\langle y(t)-v,-Au\rangle\right]. \label{duality}
\end{eqnarray}
The duality gap is a measure of optimality, in view of the following:
\begin{lemma} \label{L:limitpoints}
If $(\bar{x},\bar{y})\in\cal S$, then $\Delta_{\bar x,\bar y}(t)\ge 0$. If $\limsup_{t\to+\infty}\Delta_{u,v}(t)\le 0$, for every $(u,v)\in\mathcal{X} \times \mathcal{Y}$, then every weak subsequential limit point of $(x(t), y(t))$, as $t\to+\infty$, belongs to $\cal S$. 
\end{lemma}

\begin{proof}
If $(\bar{x},\bar{y})\in\cal S$, the subdifferential inequality and the primal-dual optimality condition applied to \eqref{duality} immediately show that $\Delta_{\bar x,\bar y}(t)\ge 0$. Assume now that $\limsup_{t\to+\infty}\Delta_{u,v}(t)\le 0$, and that $\big(x(t_k),y(t_k)\big)\rightharpoonup (x_\infty,y_\infty)$ as $k\to+\infty$. By weak lower-semicontinuity, $L(x_\infty,v)\le L(u,y_\infty)$. Writing $u=x_\infty$ and $v=y_\infty$, respectively, we see that $(x_\infty,y_\infty)$ satisfies the two inequalities in \eqref{eq:saddle_points}, whence it is a primal-dual solution.
\end{proof}

\begin{remark} \label{remark:strongly_convex}
    If the function $f$ is $\mu_f$-strongly convex, then 
    $$\Delta_{\bar x,\bar y}(t)\ge f\big(x(t)\big)-f(\bar x)+\langle A^*\bar y,x(t)-\bar x\rangle\ge \frac{\mu_f}{2}\|x(t)-\bar x\|^2.$$
    Therefore, if $\lim_{t\to+\infty}\Delta_{\bar x,\bar y}(t)= 0$, then $x(t)\to \bar x$ as $t\to+\infty$. Similar results hold for $g^*$ and $y(t)$. 
\end{remark}

Now, for each $(u,v)\in \mathcal{X}\times \mathcal{Y}$ and $t\in \J$, define the {\it anchoring function}
\begin{eqnarray}\label{anchoring}
    V_{u,v}(t) & := & \frac{\sigma(t)}{2}\|x(t)-u\|^2+\frac{\tau(t)}{2}\|y(t)-v\|^2,
\end{eqnarray}
and the {\it linear coupling}
\begin{equation}\label{eq:lin_coupl}
    M_{u,v}(t):=\scal{A(x(t)-u)}{y(t)-v},
\end{equation}
where $\sigma: \ \J \to \R_+$ and $\tau: \ \J \to \R_+$ are functions to be determined later. \\

Observe that \eqref{prec_system} is equivalent to
\begin{equation} \label{E:X_and_Y}
\left\{\begin{array}{rcccl}
X(t) & := & -\alpha(t)\dot{x}(t)-\beta(t) A^*\dot{y}(t) -A^*y(t) & \in & \partial f(x(t))\\
Y(t) & := & -\gamma(t) A\dot{x}(t)-\delta(t) \dot{y}(t)+Ax(t) & \in & \partial g^*(y(t)).
\end{array}\right.
\end{equation} 
With this in mind, for each $(u,v)\in \mathcal{X}\times \mathcal{Y}$ and $t\in \J$, we define the \emph{Bregman divergence}
\begin{equation}
\label{eq:defd}
d_{u,v}(t):= \left[f(u)-f(x(t))-\scal{X(t)}{u-x(t)}\right]+\left[g^*(v)-g^*(y(t))-\scal{Y(t)}{v-y(t)}\right],
\end{equation}
which is a non-negative quantity for every $(u,v)\in\mathcal{X}\times \mathcal{Y}$ and $t \in \J$, in view of the convexity of $f$ and $g^*$.

In what follows, when not needed for comprehension, we omit the explicit dependence on time. Equations and inequalities involving time have to be intended holding for every $t \in \J$. 

The following result plays a central role in the convergence analysis to be carried out in the next section:

\begin{lemma} \label{L:energy}
    Let $\alpha,\delta,\eta\colon[0,+\infty)\to(0,+\infty)$, $ \beta, \gamma\colon[0,+\infty) \to \R$, and set $\sigma=\alpha\eta$ and $\tau=\delta\eta$ in \eqref{anchoring}. Pick any $(u,v)\in\mathcal{X} \times \mathcal{Y}$ and any $Z\colon[0,+\infty)\to\R$. Then, the trajectory $(x,y)$ of the dynamical inclusion \eqref{prec_system} satisfies the following equality: 
\begin{align} \label{E:main_lemma}
\dot V_{u,v}+\eta Z\dot\Delta_{u,v}+\eta \Delta_{u,v} +\eta d_{u,v}  =  &\  \frac{\dot\sigma}{2}\|x-u\|^2+\frac{\dot\tau}{2}\|y-v\|^2 - \alpha\eta Z\|\dot x\|^2-\delta\eta Z\|\dot y\|^2-\eta Z(\beta+\gamma)\scal{A\dot x}{\dot y} \nonumber \\
& \ -\ \eta \big[ (\beta-Z) \scal{x-u}{ A^*\dot{y}}+ (\gamma+Z)\scal{y-v}{A\dot{x}} \big].
\end{align}
\end{lemma}

\begin{proof}
The time-derivative of $\Delta_{u,v}$ is given by
\begin{eqnarray}\label{eq: Deltaprime}
\dot \Delta_{u,v} & = & \scal{\partial f(x)}{\dot x}+\scal{v}{A\dot x}-\scal{Au}{\dot{y}} +\scal{\partial g^*(y)}{\dot{y}} \nonumber \\
& = & \scal{\partial f(x)+A^*v}{\dot x}+\scal{\partial g^*(y)-Au}{\dot{y}} \nonumber \label{E:Delta_dot_1} \\
& = & \scal{-\alpha\dot{x}-\beta A^*\dot{y}-A^*y+A^*v}{\dot x}+\scal{-\gamma A\dot{x}-\delta \dot{y}+Ax-Au}{\dot{y}} \nonumber \\
& = & -\alpha\|\dot x\|^2-\delta\|\dot y\|^2-(\beta+\gamma)\scal{A\dot x}{\dot y}+\scal{v-y}{A\dot x}+\scal{x-u}{A^*\dot y}.  \label{E:Delta_dot_2} 
\end{eqnarray}
where, in the third equality, we used system \eqref{prec_system}. Let us now compute
$$\dot V_{u,v} = \sigma\scal{x-u}{\dot x}+\tau\scal{y-v}{\dot y} +\frac{\dot\sigma}{2}\|x-u\|^2+\frac{\dot\tau}{2}\|y-v\|^2.$$
Using \eqref{prec_system}, and adopting the notation in \eqref{E:X_and_Y}, we can write the first two summands as 
\begin{eqnarray*}
\sigma\scal{x-u}{\dot x}+\tau\scal{y-v}{\dot y} & = & \frac{\sigma}{\alpha}\scal{x-u}{-\beta A^*\dot{y} - X-A^*y} + \frac{\tau}{\delta}\scal{y-v}{-\gamma A\dot{x}+Ax-Y} \\
& = & \eta \big[ \scal{u-x}{X}+ \scal{v-y}{Y} \big]\\ 
&& \ +\ \eta \big[ \scal{x-u}{-\beta A^*\dot{y}-A^*y}+  \scal{y-v}{-\gamma A\dot{x}+Ax}\big]  \\
& = & \eta \big[ \scal{u-x}{X}+ \scal{v-y}{Y}\big] - \eta\big[\big(f(u)-f(x)\big)+ \big(g^*(v)-g^*(y)\big)\big] \\
&&\ +\ \eta\big[\big(f(u)-f(x)\big)+ \big(g^*(v)-g^*(y)\big)\big] + \eta \big[ \scal{x-u}{-A^*y}+  \scal{y-v}{Ax} \big]  \\
&&\ +\ \eta \big[ \scal{x-u}{-\beta A^*\dot{y}}+ \scal{y-v}{-\gamma A\dot{x}} \big]   \\
& = & \eta \big[ \big(\scal{u-x}{X}-f(u)+f(x)\big)+\big( \scal{v-y}{Y}-g^*(v)+g^*(y)\big)\big] \\
&&\ +\ \eta\big[\big(f(u)-f(x)\big)+ \big(g^*(v)-g^*(y)\big)\big] + \eta \big[ \scal{-u}{-A^*y}+  \scal{-v}{Ax} \big]  \\
&&\ +\ \eta \big[ \scal{x-u}{-\beta A^*\dot{y}}+ \scal{y-v}{-\gamma A\dot{x}} \big]   \\
& = &  -\eta d_{u,v}-\eta\big[\big(f(x)+\scal{v}{Ax}-g^*(v) \big)  + \big(f(u)+\scal{y}{Au}-g^*(y)\big)\big]  \\
&&\ -\ \eta \big[ \beta \scal{x-u}{ A^*\dot{y}}+ \gamma\scal{y-v}{A\dot{x}} \big]   \\
& = & -\eta d_{u,v}-\eta \Delta_{u,v}- \eta \big[ \beta \scal{x-u}{ A^*\dot{y}}+ \gamma\scal{y-v}{A\dot{x}} \big].
\end{eqnarray*}

Summarizing, we have
\begin{equation} \label{E:Vdot}
    \dot V_{u,v} +\eta d_{u,v}+\eta \Delta_{u,v} = - \eta \big[ \beta \scal{x-u}{ A^*\dot{y}}+ \gamma\scal{y-v}{A\dot{x}} \big] +\frac{\dot\sigma}{2}\|x-u\|^2+\frac{\dot\tau}{2}\|y-v\|^2.
\end{equation}

Now, pick any $Z=Z(t)\in \R$, multiply \eqref{E:Delta_dot_2} by $\eta Z$, and add this to \eqref{E:Vdot}, to obtain \eqref{E:main_lemma}.
\end{proof}

\section{The linearly-constrained problem} \label{S:lin_constr}

Let $b\in\mathcal{Y}$, and suppose that $g=\iota_{\{b\}}$, so that the minimization problem~\eqref{opt_probl} becomes
\begin{equation}\label{linconstrainedproblem}
    \min_{Ax=b} f(x).
\end{equation}
In this case, $g^*(y)=\langle b,y\rangle$, and so $\partial g^*(y)\equiv b$. If we set $-\gamma=\nu>0$, system \eqref{prec_system} becomes
\begin{equation} \label{E:lin_prec_system}
    \left\{\begin{array}{rcl}
\alpha\dot{x}(t) +\beta A^*\dot{y}(t)+ \partial f(x(t))+A^*y(t) & \ni & 0\\
\nu A\dot{x}(t)-\delta \dot{y}(t)+Ax(t) & = & b.
\end{array}\right.
\end{equation}
This system was studied by Luo in \cite{luo2022primal} in the special case where $\beta\equiv 0$, $\nu\equiv 1$ and $f$ is differentiable with a Lipschitz-continuous gradient.

From Lemma \ref{L:energy}, we immediately obtain the following:
\begin{proposition}\label{P:one_sided1} 
Let $\alpha,\beta,\nu,\delta\colon[0,+\infty)\to[0,+\infty)$ be continuously differentiable, and set $\gamma=-\nu$. Assume that there is a function $\zeta\colon[0,+\infty)\to[0,+\infty)$ such that $\dot{\alpha}\leq -\zeta \alpha$ and $\dot{\delta}\leq -\zeta\delta$. Consider the anchoring function $V_{u,v}$, as in \eqref{anchoring},with $\sigma=\alpha\eta$ and $\tau=\delta\eta$ for some $\eta\colon[0,+\infty)\to[0,+\infty)$, and let $\Delta_{u,v}$ be the duality gap function defined in \eqref{duality}. For every $(\bar{x},\bar{y})\in \mathcal{S}$, we have
\begin{equation}\label{E:main_equality_onesided3}
\dot V_{\bar{x},\bar{y}}+\eta \nu\dot\Delta_{\bar{x},\bar{y}} + \zeta V_{\bar{x},\bar{y}} +\eta \Delta_{\bar{x},\bar{y}} +\eta d_{\bar{x},\bar{y}}
+ \alpha\eta \nu \|\dot x\|^2 + \delta\eta \beta\|\dot y\|^2 \leq 0. 
\end{equation}
In particular, the function $V_{\bar{x},\bar{y}}+\eta \nu {\Delta}_{\bar{x},\bar{y}}$ is nonincreasing, and the functions $\eta \Delta_{\bar{x},\bar{y}}$, $\eta \zeta V_{\bar{x},\bar{y}}$, $\eta d_{\bar{x},\bar{y}}$, $\alpha\eta \nu \|\dot x\|^2$ and $\delta\eta \beta\|\dot y\|^2$ are in $L^1$.
\end{proposition}

\begin{proof} Setting $Z=\nu=-\gamma$ in \eqref{E:main_lemma}, and using the second line of \eqref{E:lin_prec_system}, we deduce that
\begin{eqnarray} 
\dot V_{u,v} +\eta \nu\dot\Delta_{u,v}+\eta \Delta_{u,v} +\eta d_{u,v} 
& = & \frac{\dot\sigma}{2}\|x-u\|^2+\frac{\dot\tau}{2}\|y-v\|^2 -\alpha\eta \nu \|\dot x\|^2-\delta\eta \beta\|\dot y\|^2 \nonumber \\
&& \ -\ \eta (\beta-\nu)\scal{\nu A\dot x-\delta \dot y+Ax-Au}{\dot y} \nonumber \\
& = & \frac{\dot\sigma}{2}\|x-u\|^2+\frac{\dot\tau}{2}\|y-v\|^2 -\alpha\eta \nu \|\dot x\|^2-\delta\eta \beta\|\dot y\|^2  \nonumber \\
&& \ -\  \eta (\beta-\nu)\scal{b-Au}{\dot y} \nonumber \\
&  \leq & -\zeta V -\alpha\eta \nu \|\dot x\|^2-\delta\eta \beta\|\dot y\|^2 - \eta (\beta-\nu)\scal{b-Au}{\dot y}.
\label{E:linear_equality}
\end{eqnarray}
Choosing $(u,v)=(\bar{x},\bar{y})$, the statement follows.
\end{proof}

Now, given $(\bar{x},\bar{y})\in \mathcal{S}$, set
\begin{equation}\label{energy_EO}
    \mathcal E^O_{\bar{x},\bar{y}}:=V_{\bar{x},\bar{y}}+\eta \nu \Delta_{\bar{x},\bar{y}}. 
\end{equation}
From the definition of $\mathcal{E}^{O}_{\bar{x},\bar{y}}$ in \eqref{energy_EO}, it follows that
$$\dot{\mathcal E}^O_{\bar{x},\bar{y}}+\zeta\mathcal E^O_{\bar{x},\bar{y}} 
= \big(\dot V_{\bar{x},\bar{y}}+\eta \nu\dot\Delta_{\bar{x},\bar{y}} + \zeta V_{\bar{x},\bar{y}} +\eta \Delta_{\bar{x},\bar{y}}\big)
+\big(\nu\dot\eta +(\dot\nu+\zeta\nu-1)\eta\big)\Delta_{\bar{x},\bar{y}}.
$$
Therefore, if 
\begin{equation} \label{E:eta_zeta_nu}
    \nu\dot\eta +(\dot\nu+\zeta\nu-1)\eta\le 0,  
\end{equation}
it follows from \eqref{E:main_equality_onesided3} of Proposition \ref{P:one_sided1}, that
\begin{equation}\label{E:main_equality_onesided4}
\dot{\mathcal E}^O_{\bar{x},\bar{y}}+\zeta\mathcal E^O_{\bar{x},\bar{y}} \leq 0, 
\end{equation}
which immediately shows that
\begin{equation} \label{eq:ones1}  
\mathcal{E}^O_{\bar{x},\bar{y}}(t) \leq \mathcal{E}^O_{\bar{x},\bar{y}}(0)\exp\left(-\int_0^t\zeta(s)ds\right).
\end{equation}
From the definition \eqref{energy_EO} of $\mathcal{E}^O_{\bar{x},\bar{y}}$ as the sum of two nonnegative terms, we obtain
\begin{equation} \label{eq:deltao}\Delta_{\bar{x},\bar{y}}(t)\le \frac{\mathcal{E}^O_{\bar{x},\bar{y}}(0)}{\eta(t)\nu(t)}\exp\left(-\int_0^t\zeta(s)ds\right)\end{equation}
and
\begin{equation} \label{E:onesided_bounded}
    \frac{\alpha (t)}{2}\|x(t)-\bar{x}\|^2 +\frac{\delta(t)}{2}\|y(t)-\bar{y}\|^2 = V_{\bar{x},\bar{y}}(t)\leq \frac{\mathcal{E}^O_{\bar{x},\bar{y}}(0)}{\eta(t)}\exp\left(-\int_0^t\zeta(s)ds\right).
\end{equation}
Here,
\begin{eqnarray*}
\Delta_{\bar{x},\bar{y}}(t)
& = & f\big(x(t)\big)-f(\bar x)+g^*\big(y(t)\big)-g^*(\bar y) +\langle \bar y,Ax(t)\rangle-\langle y(t),A\bar x\rangle \\
& = & f\big(x(t)\big)-f(\bar x)+\langle b,y(t)-\bar y\rangle +\langle \bar y,Ax(t)\rangle-\langle y(t),b\rangle \\
& = & f\big(x(t)\big)-f(\bar x)+\langle \bar y,Ax(t)-b\rangle,
\end{eqnarray*}
so
$$\mathcal{E}^O_{\bar{x},\bar{y}}(0) = \frac{\eta_0\alpha_0}{2}\|x_0-\bar{x}\|^2 +\frac{\eta_0\delta_0}{2}\|y_0-\bar{y}\|^2 +\eta_0\nu_0\big[f(x_0)-f(\bar x)+\langle \bar y,Ax_0-b\rangle \big].$$

From the discussion above, we are naturally led to the following assumption.
\begin{assumption} \label{A:linear}
Pick a locally integrable  function $\zeta:[0,+\infty)\to[0,+\infty)$, and given constants $\alpha_0,\delta_0,\nu_0>0$. Set $\Xi(t)=\exp\left(\int_0^t\zeta(s)ds\right)$, and  define $\alpha,\delta,\nu$ by
$$\alpha(t)=\frac{\alpha_0}{\Xi(t)},\quad \delta(t)=\frac{\delta_0}{\Xi(t)} \qbox{and}\nu(t)=\frac{1}{\Xi(t)}\left[\nu_0+\int_0^t\Xi(s)ds\right].$$
Set $\gamma=-\nu$, and pick a nonnegative function $\beta$. 
\end{assumption}

We are now ready to give the main result regarding system \eqref{prec_system} and its convergence properties for the linearly constrained optimization problem \eqref{opt_probl}, with $g=\iota_{b}$.
\begin{theorem}\label{T:onesided}
    Let Assumption \ref{A:linear} hold, and let $(x(t),y(t))_{t\geq 0}$ satisfy \eqref{prec_system} with $\gamma = -\nu$ for $t>0$. Then, for every $(\bar{x},\bar{y})\in \mathcal{S}$, we have
    \begin{eqnarray}
        \Delta_{\bar{x},\bar{y}}(t) 
        & \le & \frac{\mathcal{E}^O_{\bar{x},\bar{y}}(0)}{\nu_0+\int_0^t\Xi(s)ds}, \label{eq:onesidrate1}\\
        \frac{\alpha_0}{2}\|x(t)-\bar{x}\|^2 +\frac{\delta_0}{2}\|y(t)-\bar{y}\|^2 
        & \le &  \mathcal{E}^O_{\bar{x},\bar{y}}(0) 
        , \label{eq:onesidrate2}\\
        \|Ax(t)-b\|
        & \le &  \frac{\mathcal K_0}{\nu_0+\int_0^t\Xi(s)ds}\qbox{and} \label{eq:onesidrate3}\\
        \left|f\big(x(t)\big)-f(\bar x)\right|
        & \le &  \left[\mathcal{E}^O_{\bar{x},\bar{y}}(0)+\|\bar y\|\mathcal K_0\right]\frac{1}{\nu_0+\int_0^t\Xi(s)ds},\label{eq:onesidrate4}
    \end{eqnarray}
    where
    $$\mathcal K_0=\nu_0\|Ax_0-b\|+\delta_0\|\bar y-y_0\|+\sqrt{2\delta_0\mathcal{E}^O_{\bar{x},\bar{y}}(0)}.$$
Moreover, every weak subsequential limit point $(x_*,y_*)$ of $(x(t),y(t))_{t\geq 0}$, as $t\to+\infty$, is such that $(x_*,\bar{y})$ belongs to $\mathcal {S}$. 
If in addition $f$ is $\mu$-strongly convex, then 
\begin{equation}\label{eq: estimate x strong linear}
    \norm{x(t)-\bar{x}}{}^{2} \leq \frac{2\mathcal{E}^O_{\bar{x},\bar{y}}(0)}{\mu\left(\nu_0+\int_0^t\Xi(s)ds\right)} ,
\end{equation}
and $x(t)$ converges strongly to the unique minimizer $\bar{x}$.
Similarly, if $g$ is $L$-smooth, then
\begin{equation}\label{eq: estimate y strong linear}
    \norm{y(t)-\bar{y}}{}^{2} \leq \frac{2L\mathcal{E}^O_{\bar{x},\bar{y}}(0)}{\nu_0+\int_0^t\Xi(s)ds}.
\end{equation}
\end{theorem}
\begin{proof}
Assumption \ref{A:linear} implies that
$$\dot\alpha+\zeta\alpha  =  0, \qquad 
\dot\delta+\zeta\delta  =  0\qqbox{and}
\dot\nu+\zeta\nu  =  1.
$$
Setting $\eta\equiv 1$, Proposition \ref{P:one_sided1} and inequality \eqref{E:eta_zeta_nu} are valid. Since 
$$\frac{\exp\left(-\int_0^t\zeta(s)ds\right)}{\nu(t)}=\frac{1}{\Xi(t)}\frac{1}{\frac{1}{\Xi(t)}\left[\nu_0+\int_0^t\Xi(s)ds\right]}=\frac{1}{\nu_0+\int_0^t\Xi(s)ds},$$
\eqref{eq:onesidrate1} is the same as \eqref{eq:deltao}. The definitions of $\alpha$ and $\delta$ imply that \eqref{eq:onesidrate2} is precisely \eqref{E:onesided_bounded}. Now set
$$G=\frac{\nu}{\delta}(Ax-b)-y,$$
so that, thanks to the definition of $\nu$,
$$\dot G=\left[\frac{\dot\nu\delta-\nu\dot\delta}{\delta^2}\right](Ax-b)+\frac{\nu}{\delta}A\dot x-\dot y
=\frac{1}{\delta}\big[(\dot\nu+\zeta\nu)(Ax-b)+\nu A\dot x-\delta\dot y\big]
=\frac{1}{\delta}\big[Ax-b+\nu A\dot x-\delta\dot y\big]=0.$$
Thus, $G$ is constant, and we can write
$$Ax-b=\frac{\delta}{\nu}\left[\frac{\nu_0}{\delta_0}(Ax_0-b)+y-y_0\right]=\frac{\delta}{\nu}\left[\frac{\nu_0}{\delta_0}(Ax_0-b)+y-\bar y+\bar y-y_0\right].$$
The definition of $\delta$ ensues that
$$\|Ax-b\|\le\frac{\nu_0\|Ax_0-b\|+\delta_0\|y-\bar y\|+\delta_0\|\bar y-y_0\|}{\nu\Xi},$$
which is \eqref{eq:onesidrate3}. Finally, we bound
$$\left|f\big(x(t)\big)-f(\bar x)\right|\le \Delta_{\bar{x},\bar{y}} + \|\bar y\|\|Ax-b\|,$$
and substitute accordingly to obtain \eqref{eq:onesidrate4}.
Let $(x_k,y_k)=(x(t_k),y(t_k))$ be a  subsequence of $(x(t),y(t))$ converging to $(x_*,y_*)$. Then by lower semicontinuity $L\big(x_*,\bar{y}\big)-L\big(\bar{x},y_*\big)=0$ and $Ax_*-b=0$. This implies that $(x_*,\bar{y}) \in \mathcal{S}$ (see Proposition~11 in \cite{itreg}). If $f$ is $\mu$-strongly convex or if $g$ is $L$-smooth, the corresponding estimates \eqref{eq: estimate x strong linear} and \eqref{eq: estimate y strong linear}) follow from \eqref{eq:onesidrate1} and Remark \ref{remark:strongly_convex}.
\end{proof}

\begin{remark} The convergence rates in Theorem \ref{T:onesided} depend on the choice of $\zeta$. As $\zeta \geq 0$, $\int_0^s\zeta(\tau)d\tau\geq 0$. Therefore, 
$\exp\left(\int_0^s\zeta(\tau)d\tau\right)\geq \exp(0)=1$, and so $$\frac{1}{\nu_0+\int_0^t\Xi(s)ds}\leq \frac{1}{\nu_0+t} \to 0.$$
In particular,
$$
\Delta_{\bar{x},\bar{y}}(t) 
\le  \frac{\mathcal{E}^O_{\bar{x},\bar{y}}(0)}{\nu_0+t},\quad \|Ax(t)-b\|
\le   \frac{\mathcal K_0}{\nu_0+t}\qbox{and} 
\left|f\big(x(t)\big)-f(\bar x)\right|
\le  \frac{\mathcal{E}^O_{\bar{x},\bar{y}}(0)+\|\bar y\|\mathcal K_0}{\nu_0+t}.
$$
Other choices of $\zeta$ give faster rates.
For instance, if we choose $\zeta\equiv1$ and $\nu_0=0$ so that $\nu\equiv1$, we recover the result in \cite{luo2022primal} for $\beta=0$, namely
$$\Delta_{\bar{x},\bar{y}}(t) 
\le  \mathcal{E}^O_{\bar{x},\bar{y}}(0)e^{-t},\quad \|Ax(t)-b\| \le  \mathcal K_0e^{-t}\qbox{and}
\left|f\big(x(t)\big)-f(\bar x)\right| \le \big[\mathcal{E}^O_{\bar{x},\bar{y}}(0)+\|\bar y\|\mathcal K_0\big]e^{-t}.
$$
The larger we set $\zeta$, the faster the rates. We point out, though, that a larger $\zeta$ also implies that the coefficients $\alpha$ and $\delta$ will vanish faster.
\end{remark}

\section{The asymptotically antisymmetric case}\label{S:antisym}

In this section, we discuss the {\it asymptotically} antisymmetric case where $\beta+\gamma\to 0$. In order to obtain a better insight, we first analyze the {\it exact} antisymmetric case $\beta+\gamma\equiv 0$, under the following hypotheses:

\begin{assumption}[Parameter functions] \label{A:antisymmetric}
Pick any differentiable function $\beta:[0,+\infty)\to(0,+\infty)$, set $\beta_0=\beta(0)$, and then define
$$\alpha(t) := \frac{\alpha_0}{\beta_0}\,\beta(t)\,\exp\left(-\int_0^t\frac{ds}{\beta(s)}\right),\qquad \gamma:=-\beta  \qqbox{and}
    \delta(t) := \frac{\delta_0}{\beta_0}\,\beta(t)\,\exp\left(-\int_0^t\frac{ds}{\beta(s)}\right),$$
with $\alpha_0,\delta_0>0$. Finally, consider the energy-function $\mathcal E^A_{u,v}:=\Delta_{u,v}+V_{u,v}$, where $\Delta$ and $V$ are given by \eqref{duality} and \eqref{anchoring}, respectively.
\end{assumption}

\begin{remark}
    If $\dot\beta\equiv 1$, then $\alpha\equiv\alpha_0$ and $\delta\equiv\delta_0$. In turn, if $\dot\beta\ge 1$, then $\alpha\ge\alpha_0$ and $\delta\ge\delta_0$.
\end{remark}

\begin{theorem} \label{T:antisymmetric}
Let Assumption \ref{A:antisymmetric} hold, and let $(x(t),y(t))_{t\geq 0}$ satisfy \eqref{prec_system} for $t>0$. Fix $(u,v)\in \mathcal{X}\times\mathcal{Y}$ and $(\bar{x},\bar{y})\in \mathcal{S}$. Then, the functions $\frac{1}{\beta} d_{\bar{x},\bar{y}}$, $\alpha \|\dot x\|^2$ and $\delta \|\dot y\|^2$ belong to $L^1$, and
\begin{eqnarray}
    \Delta_{u,v}(t)
    & \le & \mathcal E^A_{u,v}(0)\exp\left(-\int_0^t\frac{ds}{\beta(s)}\right) \label{E:antisymmetric_1} \\
    f\big(x(t)\big)-f(\bar x)+\langle A^*\bar y,x(t)-\bar x\rangle 
    & \le & \mathcal E^A_{\bar{x},\bar{y}}(0)\exp\left(-\int_0^t\frac{ds}{\beta(s)}\right) \label{E:antisymmetric_2} \\
    g^*\big(y(t)\big)-g^*(\bar y)+\langle y(t)-\bar y,-A\bar x\rangle
    & \le & \mathcal E^A_{\bar{x},\bar{y}}(0)\exp\left(-\int_0^t\frac{ds}{\beta(s)}\right).\label{E:antisymmetric_3} 
\end{eqnarray}
Moreover, the trajectory is bounded, with
\begin{equation} \label{E:antisymmetric_4}
    \frac{\alpha_0}{2}\|x(t)-\bar x\|^2+\frac{\delta_0}{2}\|y(t)-\bar y\|^2
\le \beta_0\mathcal E^A_{\bar x,\bar y}(0).
\end{equation}
In particular, if $\int_0^\infty\frac{1}{\beta(s)}ds=\infty$, then every weak subsequential limit point of $(x(t),y(t))$, as $t\to+\infty$, is a primal-dual solution of \eqref{opt_probl}. If we further assume that $f$ is $\mu$-strongly convex, then $x(t)$ converges strongly, as $t\to+\infty$, to the unique $\bar x$, with
$$\|x(t)-\bar x\|^2\le \frac{2}{\mu}\,\mathcal E^A_{\bar x,\bar y}(0)\exp\left(-\int_0^t\frac{ds}{\beta(s)}\right).$$
Similarly, if $g$ is $L$-smooth, then
$$\|y(t)-\bar y\|^2 \le 2L\,\mathcal E^A_{\bar x,\bar y}(0)\exp\left(-\int_0^t\frac{ds}{\beta(s)}\right).$$
\end{theorem}

\begin{proof}
Take $(u,v)\in\mathcal{X} \times \mathcal{Y}$, and set $\eta=1/\beta$, so that $\sigma=\alpha/\beta$ and $\tau=\delta/\beta$ in \eqref{anchoring}. Lemma \ref{L:energy} gives
$$ 
\dot V_{u,v}+\dot\Delta_{u,v}+\frac{1}{\beta} \Delta_{u,v} +\frac{1}{\beta} d_{u,v} 
= \frac{\dot\sigma}{2}\|x-u\|^2+\frac{\dot\tau}{2}\|y-v\|^2 - \alpha \|\dot x\|^2-\delta \|\dot y\|^2 \le\frac{1}{\beta}V_{u,v}- \alpha \|\dot x\|^2-\delta \|\dot y\|^2.
$$
Since
$\dot{\mathcal E}^A_{u,v}=\dot V_{u,v}+\dot\Delta_{u,v}$,
we deduce that
\begin{equation} \label{E:EG'_nonpositive3}
   \dot{\mathcal E}^A_{u,v}+\frac{1}{\beta} \mathcal E^A_{u,v}+\frac{1}{\beta} d_{u,v} + \alpha \|\dot x\|^2 + \delta \|\dot y\|^2\le 0,
\end{equation}
which implies the integrability conditions, plus
\begin{equation} \label{E:general_energy_to_zerobefore}
    V_{u,v}(t)+\Delta_{u,v}(t)=\mathcal E^A_{u,v}(t)\le \mathcal E^G_{u,v}(0)\exp\left(-\int_0^t\frac{1}{\beta(s)}ds\right).
\end{equation}
This estimation gives \eqref{E:antisymmetric_1}, \eqref{E:antisymmetric_2}, \eqref{E:antisymmetric_3} and \eqref{E:antisymmetric_4}. In particular,
\begin{equation} \label{E:general_gap_to_zero}
    \Delta_{u,v}(t)\le \mathcal E^A_{u,v}(0)\exp\left(-\int_0^t\frac{1}{\beta(s)}ds\right),
\end{equation}
and so
$$
\limsup_{t\to+\infty}\Delta_{u,v}(t)\le 0.
$$
In view of Lemma \ref{L:limitpoints}, weak subsequential limit points are primal-dual solutions. The remaining statements follow from Remark \ref{remark:strongly_convex}. 
\end{proof}

\begin{remark}
If $\beta\equiv 1$, $\gamma\equiv -1$ and $\eta\equiv 1$, then $\Delta_{u,v}(t)\le \mathcal E^G_{u,v}(0)e^{-t}$ for every $(u,v)\in \mathcal{X}\times\mathcal{Y}$, and
$$
\begin{array}{rcccl}
0 & \le & f\big(x(t)\big)-f(\bar x)+\langle A^*\bar y,x(t)-\bar x\rangle 
    & \le & \mathcal E^G_{\bar x,\bar y}(0)e^{-t} \smallskip \\
0 & \le & g^*\big(y(t)\big)-g^*(\bar y)+\langle y(t)-\bar y,-A\bar x\rangle
    & \le & \mathcal E^G_{\bar x,\bar y}(0)e^{-t}
\end{array}$$ 
for every $(\bar{x},\bar{y})\in \mathcal{S}$.
\end{remark}

\subsection*{A perturbed version}
We are now in a position to study a more general setting, under very similar hypotheses, namely:

\begin{assumption}[Parameter functions] \label{A:general}
Pick differentiable functions $\beta,\gamma:[0,+\infty)\to\R$, such that $\beta(t)>\gamma(t)$ for every $t\ge 0$. Set $Z=\frac{\beta-\gamma}{2}$ and $Z_0=Z(0)$, and define $\alpha,\delta:[0,+\infty)\to(0,+\infty)$ by
$$\alpha(t)=\frac{1}{\Xi(t)}\left[\alpha_0-\sqrt{\frac{\alpha_0}{\delta_0}}\int_0^t\Xi(s)W(s)\,ds\right] \qqbox{and} \delta(t)=\frac{1}{\Xi(t)}\left[\delta_0-\sqrt{\frac{\delta_0}{\alpha_0}}\int_0^t\Xi(s)W(s)\,ds\right],$$
where
$$W(t)=\frac{\|A\|}{2}\,\frac{|\beta(t)+\gamma(t)|}{\beta(t)-\gamma(t)} \qqbox{and} \Xi(t)=\exp\left(\int_0^t\frac{1-\dot Z(s)}{Z(s)}\,ds\right)=\frac{Z_0}{Z(t)}\,\exp\left(\int_0^t\frac{ds}{Z(s)}\right).$$
with $\alpha_0,\delta_0>0$. We assume, moreover, that $\|A\||\beta+\gamma|\le\sqrt{\alpha\delta}$, which is equivalent to
\begin{equation} \label{E:A_general}
    4Z(t)\Xi(t)W(t)+\int_0^t\Xi(s)W(s)\,ds\le\sqrt{\alpha_0\delta_0},
\end{equation}
for all $t\ge 0$, and ensures that $\alpha,\delta>0$.  
\end{assumption}

\begin{remark}
If $\beta+\gamma\equiv 0$, then $W\equiv 0$ and $Z=\beta$. The definitions of $\alpha$ and $\delta$ coincide with the ones in Assumption \ref{A:antisymmetric}.   
\end{remark}

\begin{theorem} \label{T:general_theorem}
    Let Assumption \ref{A:general} hold, and let $(x(t),y(t))_{t\geq 0}$ satisfy \eqref{prec_system} for $t>0$. Then, for every $(u,v)\in \mathcal{X}\times\mathcal{Y}$ and $(\bar{x},\bar{y})\in \mathcal{S}$, we have
    \begin{eqnarray*}
        \Delta_{u,v}(t)
        & \le & D_{u,v}(0)\exp\left(-\int_0^t\frac{ds}{Z(s)}\right) \\
        f\big(x(t)\big)-f(\bar x)+\langle A^*\bar y,x(t)-\bar x\rangle 
        & \le & D_{\bar x,\bar y}(0)\exp\left(-\int_0^t\frac{ds}{Z(s)}\right) \\
        g^*\big(y(t)\big)-g^*(\bar y)+\langle y(t)-\bar y,-A\bar x\rangle
        & \le & D_{\bar x,\bar y}(0)\exp\left(-\int_0^t\frac{ds}{Z(s)}\right), 
    \end{eqnarray*}
    where 
    $$D_{u,v}(0)=\frac{\alpha_0}{2Z_0}\|x_0-u\|^2+\frac{\delta_0}{2Z_0}\|y_0-v\|^2+\Delta_{u,v}(0).$$
    In particular, if $\int_0^\infty\frac{1}{Z(s)}ds=\infty$, then every weak subsequential limit point of $(x(t),y(t))$, as $t\to+\infty$, is a primal-dual solution of \eqref{opt_probl}. If, moreover, $\int_0^\infty\Xi(s)W(s)\,ds<\sqrt{\alpha_0\delta_0}$, the trajectory remains bounded. Finally, if we further assume that $f$ is $\mu$-strongly convex, then $x(t)$ converges strongly, as $t\to+\infty$, to the unique $\bar x$, with
    $$\|x(t)-\bar x\|^2\le \frac{2D_{\bar x,\bar y}(0)}{\mu}\exp\left(-\int_0^t\frac{ds}{Z(s)}\right).$$
    Similarly, if $g$ is $L$-smooth, then
    $$\|y(t)-\bar y\|^2 \le 2LD_{\bar x,\bar y}(0)\exp\left(-\int_0^t\frac{ds}{Z(s)}\right).$$
\end{theorem}

\begin{proof}
Take $(u,v)\in\mathcal{X} \times \mathcal{Y}$. Set $\eta=\Xi$, so that $\sigma=\alpha\Xi$ and $\tau=\delta\Xi$ in \eqref{anchoring}. Writing $Z=\frac{\beta-\gamma}{2}>0$, Lemma \ref{L:energy} gives
\begin{eqnarray*}
    \dot V_{u,v}+\Xi Z\dot\Delta_{u,v}+\Xi \Delta_{u,v} +\Xi d_{u,v} 
    & = & \frac{\dot\sigma}{2}\|x-u\|^2+\frac{\dot\tau}{2}\|y-v\|^2 - \Xi Z\alpha \|\dot x\|^2-\Xi Z\delta \|\dot y\|^2 \\
    &&\ -\ \frac{\Xi(\beta+\gamma)}{2}\big[\scal{x-u}{ A^*\dot{y}}+ \scal{y-v}{A\dot{x}}+2Z\scal{A\dot x}{\dot y}\big].
\end{eqnarray*}
Using Young's inequality repeatedly, we can bound the terms in the second line, to deduce that
\begin{eqnarray*}
    \dot V_{u,v}+\Xi Z\dot\Delta_{u,v}+\Xi \Delta_{u,v} +\Xi d_{u,v} 
    & \le & \left[\frac{\dot\sigma}{2}+ \frac{\Xi|\beta+\gamma|\|A\|}{4\varepsilon_1}\right]\|x-u\|^2+\left[\frac{\dot\tau}{2}+ \frac{\Xi|\beta+\gamma|\|A\|}{4\varepsilon_2}\right]\|y-v\|^2  \\
    &&\ +\ \Xi\left[\frac{|\beta+\gamma|\|A\|(\varepsilon_2+2Z\varepsilon_3^{-1})}{4}-\alpha Z\right] \|\dot x\|^2 \\
    &&\ +\ \Xi\left[\frac{|\beta+\gamma|\|A\|\big(\varepsilon_1+2Z\varepsilon_3\big)}{4}-\delta Z\right]\|\dot y\|^2
\end{eqnarray*}
for every $\varepsilon_1,\varepsilon_2,\varepsilon_3>0$. By writing $M=|\beta+\gamma|\|A\|$, and setting $\varepsilon_1=2Z\varepsilon_3$ and $\varepsilon_2=2Z\varepsilon_3^{-1}$, we get
\begin{eqnarray*}
    \dot V_{u,v}+\Xi Z\dot\Delta_{u,v}+\Xi \Delta_{u,v} +\Xi d_{u,v} 
    & \le & \left[\frac{\dot\sigma}{2}+ \frac{\Xi M}{8Z\varepsilon_3}\right]\|x-u\|^2+\left[\frac{\dot\tau}{2}+ \frac{\Xi M\varepsilon_3}{8Z}\right]\|y-v\|^2  \\
    &&\ +\ \Xi Z\left[M\varepsilon_3^{-1}-\alpha \right] \|\dot x\|^2 + \Xi Z\left[M\varepsilon_3-\delta \right]\|\dot y\|^2
\end{eqnarray*}
Now, write $\omega=M/\sqrt{\alpha\delta}\in[0,1]$, so that $M=\omega\sqrt{\alpha\delta}$, and
\begin{eqnarray*}
    \dot V_{u,v}+\Xi Z\dot\Delta_{u,v}+\Xi \Delta_{u,v} +\Xi d_{u,v} 
    & \le & \left[\frac{\dot\sigma}{2}+ \frac{\Xi\omega \sqrt{\alpha\delta}}{8Z\varepsilon_3}\right]\|x-u\|^2+\left[\frac{\dot\tau}{2}+ \frac{\Xi\omega \sqrt{\alpha\delta}\varepsilon_3}{8Z}\right]\|y-v\|^2  \\
    &&\ +\ \Xi Z\left[\omega\sqrt{\alpha\delta}\varepsilon_3^{-1}-\alpha \right] \|\dot x\|^2 + \Xi Z\left[\omega\sqrt{\alpha\delta}\varepsilon_3-\delta \right]\|\dot y\|^2.
\end{eqnarray*}
Since $\omega\le 1$, the choice $\varepsilon_3=\sqrt{\delta/\alpha}$ makes the last two coefficients non-positive. We conclude that
\begin{eqnarray*}
\dot V_{u,v}+\Xi Z\dot\Delta_{u,v}+\Xi \Delta_{u,v} +\Xi d_{u,v} 
& \le & \left[\dot\sigma+ \frac{\omega\sigma}{4Z}\right]\frac{\|x-u\|^2}{2}+\left[\dot\tau + \frac{\omega\tau}{4Z}\right]\frac{\|y-v\|^2}{2} \\
&&\ -(1-\omega)\alpha\Xi Z\|\dot x\|^2-(1-\omega)\delta\Xi Z\|\dot x\|^2,
\end{eqnarray*}
since $\sigma=\alpha\Xi$ and $\tau=\delta\Xi$. The definitions of $\alpha$ and $\delta$ in Assumption \ref{A:general} imply that
$$\dot\sigma+ \frac{\omega\sigma}{4Z}=\dot\tau+ \frac{\omega\tau}{4Z}=0.$$
Indeed, we have
$$\dot\sigma=\frac{d}{dt}\big(\Xi\alpha\big) =-\sqrt{\frac{\alpha_0}{\delta_0}}\,\Xi W,\qbox{while}\frac{\omega\sigma}{4Z}=\frac{\left[\frac{M}{\sqrt{\alpha\delta}}\right]\big[\Xi\alpha\big]}{4Z}=\sqrt{\frac{\alpha}{\delta}}\,\Xi\,\frac{M}{4Z}=\sqrt{\frac{\alpha_0}{\delta_0}}\,\Xi W=-\dot\sigma,$$
and similarly for $\tau$. It follows that
\begin{equation} \label{E:EG'_nonpositive1}
\dot V_{u,v}+\Xi Z\dot\Delta_{u,v}+\Xi \Delta_{u,v} +R_{u,v}\le 0,
\end{equation}
where we have written 
$$R_{u,v}:=\Xi d_{u,v} +(1-\omega)\Xi Z\big(\alpha\|\dot x\|^2+\delta\|\dot y\|^2\big)\ge 0.$$ 
Defining 
$\mathcal E^G_{u,v}:=V_{u,v}+\Xi Z\Delta_{u,v}$, we get
$$\dot{\mathcal E}^G_{u,v}=\dot V_{u,v}+\Xi Z\dot\Delta_{u,v}+(\dot\Xi Z+\Xi \dot Z)\Delta_{u,v}=\dot V_{u,v}+\Xi Z\dot\Delta_{u,v}+\Xi\Delta_{u,v},
$$
in view of the definition of $\Xi$. Therefore, \eqref{E:EG'_nonpositive1} is equivalent to
$$   \dot{\mathcal E}^G_{u,v}+R_{u,v}\le 0.$$
Since $R_{u,v}\ge 0$, it follows from Gr\"onwall's Inequality that
\begin{equation} \label{E:general_energy_to_zero}
    V_{u,v}(t)+\Xi(t) Z(t)\Delta_{u,v}(t)=
    \mathcal E^G_{u,v}(t)\le \mathcal E^G_{u,v}(0).
\end{equation}
In particular,
$$    \Delta_{u,v}(t)\le \frac{\mathcal E^G_{u,v}(0)}{\Xi(t) Z(t)}=D_{u,v}(0)\exp\left(-\int_0^t\frac{ds}{Z(s)}\right),
$$
which is the first claimed estimation, and implies the next two. Since $\frac{1}{Z}\notin L^1$, we have
$$
\limsup_{t\to+\infty}\Delta_{u,v}(t)\le 0.
$$
Lemma \ref{L:limitpoints} then shows that the weak subsequential limit points are primal-dual solutions.
Now, if $(\bar x,\bar y)\in\mathcal S$, then 
$\Delta_{\bar x,\bar y}\ge 0$. Therefore,  \eqref{E:general_energy_to_zero} gives
$$\frac{\alpha(t)\Xi(t)}{2}\|x(t)-\bar x\|^2+\frac{\delta(t)\Xi(t)}{2}\|y(t)-\bar y\|^2
\le \mathcal E^G_{\bar x,\bar y}(0),$$
by the definition of $V$. If $\int_0^\infty\Xi(s)W(s)\,ds<\sqrt{\alpha_0\delta_0}$, then $\alpha\Xi$ and $\delta\Xi$ are bounded from below by a positive constant, so the trajectory $(x,y)$ must remain bounded. As in Theorem \ref{T:antisymmetric}, the remaining statements concerning the $\mu$-strong convexity of $f$ (and the $L$-smoothness of $g$ respectively) follow easily from Remark \ref{remark:strongly_convex}
\end{proof}

\begin{remark}
For Assumption \ref{A:general} to hold, $\beta-\gamma$ cannot go to infinity too fast (if it does go to infinity), in order for $\frac{1}{\beta-\gamma}$ not to be integrable on $(0,+\infty)$. On the other hand, the function
$$
\frac{|\beta+\gamma|}{\big(\beta-\gamma\big)^2}\exp\left(\int\frac{2}{\beta-\gamma}\right)$$
must go to zero sufficiently fast to ensure \eqref{E:A_general}. The two conditions can be obtained with $\beta-\gamma\sim t$ and $\beta+\gamma\sim\frac{1}{t+1}$, or also if $\beta-\gamma\equiv 2Z_0$ and $|\beta+\gamma|\sim e^{-\kappa t}$ with $\kappa Z_0>1$, adjusting the constants appropriately.
\end{remark}

\section{Numerical experiments} \label{S:num}

In this section we illustrate the behavior of system  \eqref{prec_system} on some synthetic examples. We test numerically two instances of system \eqref{prec_system}, namely the antisymmetric case (i.e.  \eqref{prec_system} with $\gamma=-\beta$) and the symmetric one \eqref{prec_system_simm} discussed in the Appendix (i.e. \eqref{prec_system} with $\gamma = \beta$). In Example~\ref{ex:linconstr1}, that is a linearly constrained problem (see \eqref{linconstrainedproblem}), we also test a triangular preconditioner, with $\beta=0$. All the experiments below were performed in the programming language Julia, by using the ODE solver “Tsit5” with absolute tolerance $\approx 10^{-12}$. 

\begin{example}[Quadratic minimization]
    In this first example, we consider the minimization problem \eqref{opt_probl} in $\R^{d}$, with $d=100$, $A\in\R^{d\times d}$ a random matrix, $f(x) = \frac{1}{2}\ps{Bx}{x}$, $g(y)=\frac{1}{2}\ps{Cy}{y}$, where $B\in \R^{d\times d}$ is a positive semi-definite matrix with a non-trivial kernel ($10\%$ of the eigenvalues are chosen to be equal to zero) and $C\in R^{d\times d}$ is positive definite (thus invertible). In this case, we compare the convergence of trajectories in terms of the duality gap $\Delta_{\bar{x},\bar{y}}$ defined in \eqref{eq:duality_gap} for: the antisymmetric system (\eqref{prec_system} with $\gamma=-\beta$), the symmetric system \eqref{prec_system_simm}, and the classical Arrow-Hurwicz method (which corresponds to system \eqref{prec_system_simm}, without preconditioning, i.e. $\beta=\gamma=0$ and $\alpha=\delta =1$) . In Figure \ref{Figure4} we can observe that the preconditioning corresponding to the antisymmetric system leads to a faster decay for the duality gap $\Delta_{\bar{x},\bar{y}}$, with respect to the Arrow-Hurwicz system, as also to the symmetric system \eqref{prec_system_simm}. 

\begin{figure}[t]
	\begin{center}
		\includegraphics[trim=0cm 4mm 0mm 1mm, scale=0.33]{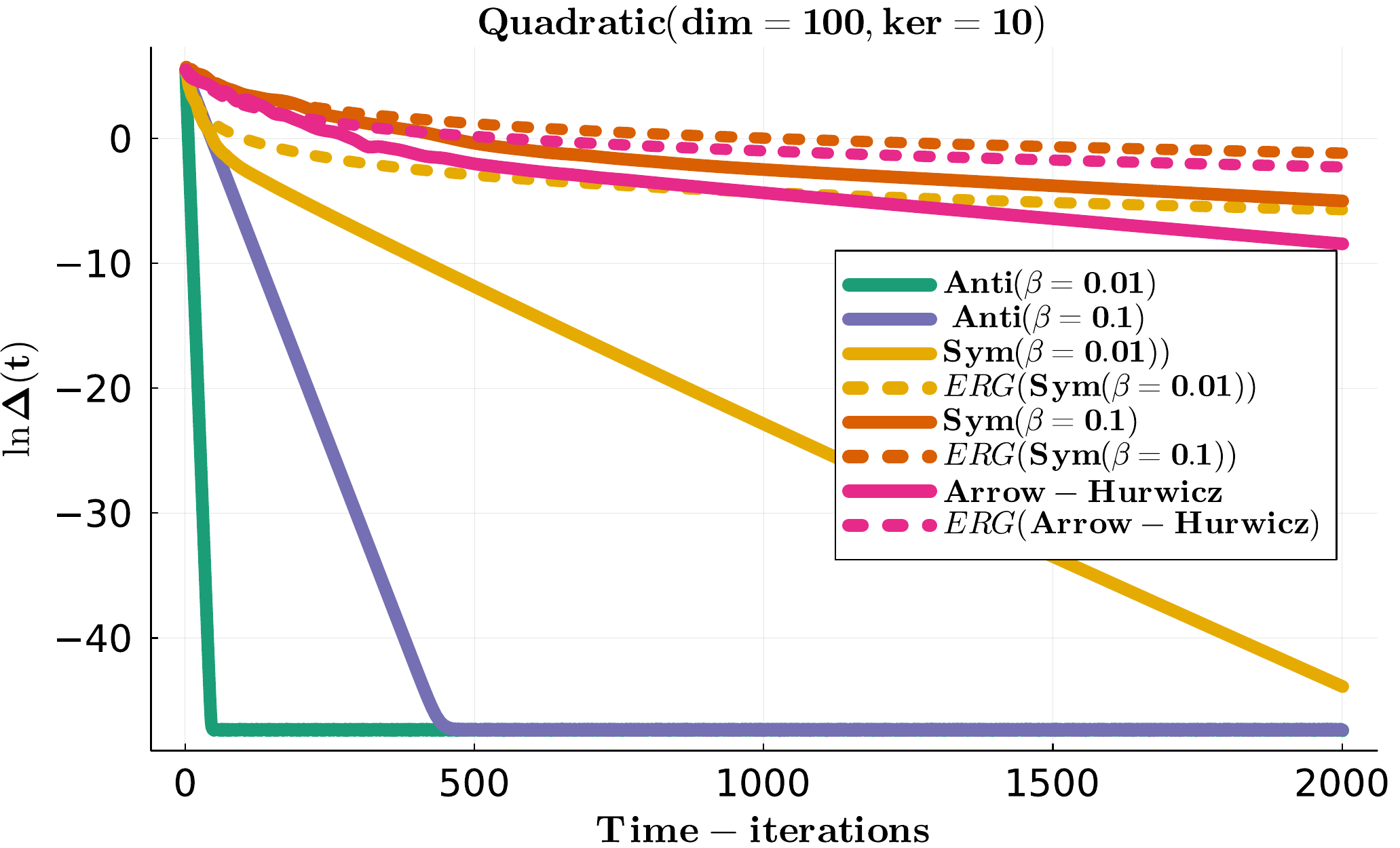}
	\end{center}
	\caption{Comparison of the antisymmetric and symmetric system, with two different choices for the parameter $\beta$ and the standard Arrow-Hurwicz method for $100$-dimensional quadratic minimization problem ($f(x) = \ps{Bx}{x}/2$, $g(x)=\ps{Cx}{x}/2$). The green and light-purple lines correspond to the trajectories of the antisymmetric system with $\beta=0.01$ and $\beta=0.1$ (resp) $\alpha(t)=\delta(t)=\beta\exp\left(-\frac{t}{\beta}\right)$ in each case, according to Assumption \ref{A:antisymmetric}.
		In yellow and orange the trajectories generated by the symmetric system with  with $\beta=0.01$ and $\beta=0.1$ and $\alpha = \beta\norm{A}{}+0.001$ in each case and their corresponding ergodic ones in dotted lines. 
		In magenta the trajectory corresponding to  the Arrow-Hurwicz method (system \eqref{prec_system} without preconditioning, i.e. $\beta = \gamma =0$ and $\alpha = \delta =1$). }
\label{Figure4}
\end{figure}
\end{example}

\begin{example}[Linear constrained quadratic]\label{ex:linconstr1}
In this example, we consider problem \eqref{opt_probl} in $\R^{d}$, with $d=50$, $A\in \R^{d\times d}$ a random matrix, $f(x) = \frac{1}{2}\ps{Bx}{x}$, where $B\in\R^{d\times d}$ is a random positive semi-definite matrix with non-empty kernel ($10\%$ of the eigenvalues equal to $0$) and $g(x)=\iota_{\{b\}}(y)$ is the indicator function, where $b$ is a random vector generated to be in the range of $A$. 
In this case we compare the performance of the trajectories of the antisymmetric system (i.e. \eqref{prec_system} with $\gamma=-\beta$), the symmetric one \eqref{prec_system_simm} and the ones generated by the triangular system \eqref{E:lin_prec_system}, in terms of the langrangian gap $\Delta_{\bar{x},\bar{y}}$. We consider two different choices for the parameter $\beta$ ($\beta=0.1$ and $\beta =0.5$) for the antisymmetric and symmetric system and two choices for the parameter $\nu$ ($\nu=0.1$ and $\nu =0.5$) for the triangular system. The corresponding results are reported to Figure \ref{Figure5}. The convergence of the duality gap along the trajectories of the antisymmetric system is much faster than the one along the the trajectories of symmetric and the triangular systems in this setting.

\begin{figure}[t]
\begin{center}
	\includegraphics[trim=0cm 4mm 0mm 1mm, scale=0.33]{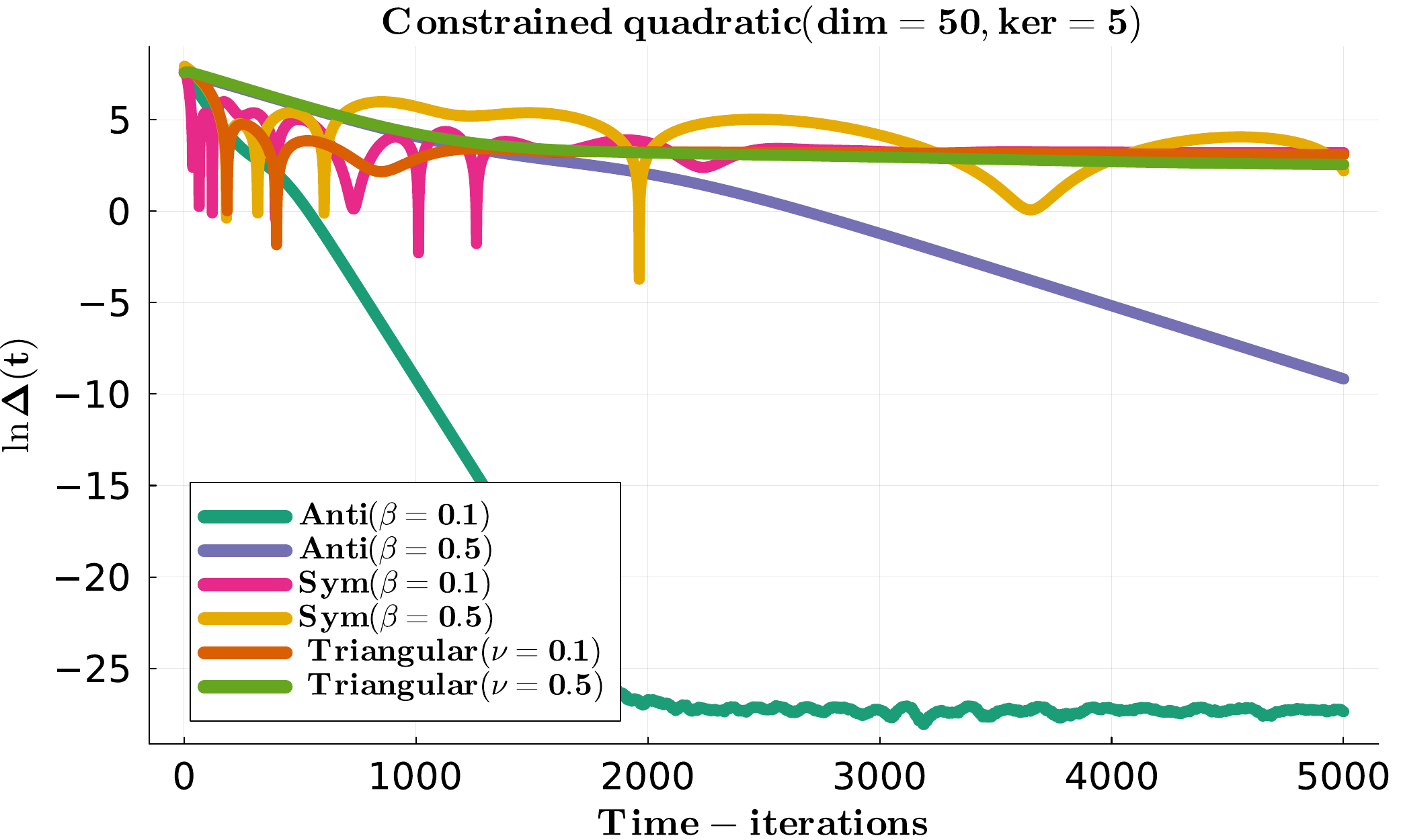}
\end{center}
\caption{Comparison of the antisymmetric with the symmetric and the triangular ODE system for $50$-dimensional quadratic minimization problem ($f(x) = \ps{Bx}{x}/2$, $g(y)=\iota_{\{b\}}(y)$). The dark green and light purple lines correspond to the antisymmetric case with $\beta=0.1$ and $\beta=0.5$ (resp.), with $\alpha(t)=\delta(t)=\beta\exp\left(-\frac{t}{\beta}\right)$ in each case, while the magenta and yellow one correspond to the symmetric system with $\beta=0.1$ and $\beta=0.5$ (resp.) and $\alpha=\beta\norm{A}{}+0.001$ for each case. In orange and light green the trajectories generated by the triangular system \eqref{E:lin_prec_system} with $\nu=0.1$ and $\nu=0.5$ (resp.) and $\beta =0 $ as considered in \cite{luo2022primal}.}\label{Figure5}
\end{figure}
\end{example}

\begin{example}[LASSO]
    In this example we consider the LASSO problem $\underset{x\in \R^{40}}{\min}\frac{1}{2} \norm{Ax-b}{}^{2}  + \lambda \norm{x}{1}$, where $A\in \R^{20\times 40}$ is a random matrix  and $b=Ax_{\ast}+\epsilon$, where $\epsilon$ is a a Gaussian centered at the origin with variance $0.05$. 
In this problem we set $\lambda=0.005$ and we compare the solution-trajectories of the symmetric and antisymmetric system in terms of the duality gap $\Delta_{\bar{x},\bar{y}}(\cdot,\cdot)$. Here we consider $\bar{x}$ as the output of the standard Forward-Backward algorithm applied to the the LASSO problem (see e.g. \citet{combettes2005signal}), after $10^5$ iterations, as a good approximation of a true solution of $\underset{x\in \R^{40}}{\min}\frac{1}{2}\norm{Ax-b}{}^{2} + \lambda\norm{x}{1}$,
while $\bar{y}=A\bar{x}-b$. Due to long computational time of the ODE solver (Tsit$5$), we additionally included a cap of $30$ minutes of overall time as a stopping criterion to each system's solver (symmetric \& antisymmetric). The corresponding results are reported to Figure \ref{Figurelasso}. Similarly to the previous examples,  Figure \ref{Figurelasso} shows a faster convergence of the antisymmetric system. 

\begin{figure}[t]
	\begin{center}
    	\includegraphics[trim=0cm 4mm 0mm 1mm, scale=0.33]{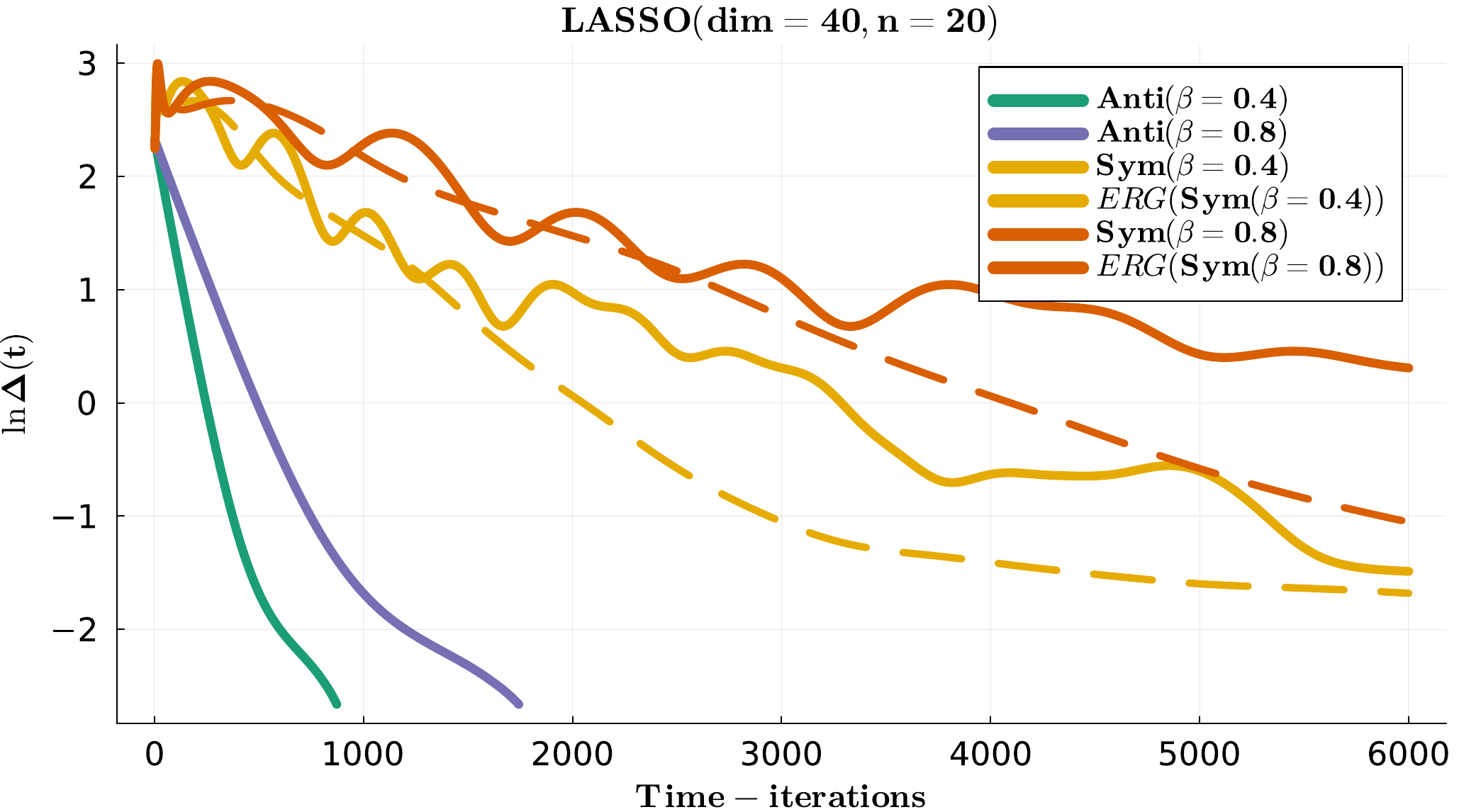}
	\end{center}
	\caption{Comparison of the trajectories generated by the antisymmetric and symmetric system for the LASSO problem, in terms of Langrangian gap $\Delta_{\bar{x},\bar{y}}(t)$. The plots in dark green and light purple color correspond to the antisymmetric system with $\beta = 0.4$ and $\beta=0.8$ respectively, with $\alpha(t)=\delta(t)=\beta\exp\left(-\frac{t}{\beta}\right)$ in each case. In solid yellow and orange color the ones corresponding to the symmetric system with $\beta=0.4$ and $\beta=0.8$ (resp.) and $\alpha = \beta\norm{A}{}+0.005$ and in dashed (yellow and orange) lines their ergodic versions (resp.).}\label{Figurelasso}
\end{figure}
\end{example}

\section{Concluding remarks}\label{S:conclusions}

We have studied a continuous-time dynamic model for primal-dual splitting in the form of a preconditioned saddle-point flow. When applied to linearly constrained optimization problems, the proposed dynamics yield fast convergence rates for the primal-dual gap, which in turn imply rates for both the optimality and the feasibility gaps. Remarkably, these convergence rates are largely insensitive to the choice of non-diagonal preconditioner coefficients. For the general problem \eqref{opt_probl}, we found that symmetric preconditioners do not significantly improve convergence properties beyond what is induced by a time-rescaling. In contrast, antisymmetric preconditioners allow for the convergence of the non-ergodic primal-dual gap, and we establish corresponding non-ergodic convergence rates. Numerical simulations further suggest that antisymmetric preconditioning contributes to improved stabilization of the system. 

Our results raise several questions: 
First, under antisymmetric preconditioning, the convergence of the non-ergodic primal-dual gap implies that weak subsequential limits points of the non-ergodic trjectories are primal-dual solutions, it is still unclear whether the trajectories themselves do converge. A sufficient condition for this would be that $\lim_{t\to\infty}\|(x(t),y(t))-(\bar x,\bar y)\|$ exists, for every $(\bar x,\bar y)\in\cal S$. Second, it would be interesting to investigate whether other combinations of parameters, beyond the asymptotically antisymmetric ones, yield non-ergodic convergence, in the case of the 
general problem \eqref{opt_probl}. Next, our analysis is limited to problems with bilinear coupling. Extending the theory to more general saddle-point problems, involving nonlinear convex-concave Lagrangians remains an open challenge. In such cases, it is not clear how to design appropriate preconditioners.
Finally, in the presence of multiple primal-dual solutions, the question of whether preconditioning induces a bias in solution selection deserves attention.

\appendix

\section{The symmetric case $\gamma=\beta$} \label{S:symmetric}

In this section we analyze system \eqref{prec_system} under the parameter choice $\gamma=\beta$, which leads to the following formulation
\begin{equation} \label{prec_system_simm}
\left\{\begin{array}{rcl}
\alpha(t)\dot{x}(t)+\beta(t) A^*\dot{y}(t) + \partial f(x(t))+A^*y(t) & \ni & 0\\
\beta(t) A\dot{x}(t)+\delta(t) \dot{y}(t)-Ax(t)+\partial g^*(y(t)) & \ni & 0.
\end{array}\right.
\end{equation} 

\begin{remark}
    If $\alpha$ and $\delta$ are constant multiples of $\beta$, this system is a time reparameterization of a modified Arrow-Hurwicz system, as we explain in the proof of Theorem \ref{T:sym_rates}, part v). The convergence rates given below hold under more general conditions.
\end{remark}

We establish convergence results for the ergodic trajectory, defined as follows
$$\hat{x}(t):=\frac{1}{\int_0^t 1/\beta(s) \ ds} \int_0^t x(s)/\beta(s);\quad \hat{y}(t):=\frac{1}{\int_0^t 1/\beta(s) \ ds}  \int_0^t y(s)/\beta(s)\ ds.$$
The duality gap associated with the ergodic trajectory is denoted by
$$\hat{\Delta}_{u,v}(t):=L(\hat{x}(t),v)-L(u,\hat{y}(t)).$$

In this setting, the Lyapunov function is constructed from the anchoring function and the linear coupling introduced in \eqref{anchoring} and \eqref{eq:lin_coupl} in Section~\ref{S:estimates}:
\begin{equation} \label{E:parameters_symmetric}
\mathcal E^S_{u,v}:=V_{u,v}+ M_{u,v},
\quad \text{with}\quad \sigma=\frac{\alpha}{\beta} \quad\text{and\quad} \tau=\frac{\delta}{\beta}
\end{equation}
(so $\eta=1/\beta$). We next establish convergence properties and explicit convergence rates for the duality gap along the ergodic trajectory, together with the corresponding estimates for the associated Bregman divergence \eqref{eq:defd}.
\begin{theorem}\label{T:sym_rates}
Let the functions $\alpha, \beta$ and $\delta$ in \eqref{prec_system_simm} be absolutely continuous and such that $\beta(t)>0$ for every $t\geq 0$. Assume also that $\frac{\alpha}{\beta}$ and $\frac{\delta}{\beta}$ are nonincreasing, and that
\begin{equation} \label{eq:negder}
c:=\min\left\{\inf_{t\ge 0} \frac{\alpha(t)}{\beta(t)},\inf_{t\ge 0} \frac{\delta(t)}{\beta(t)}\right\}  >  \|A\|.  
\end{equation}
Let $(x,y)$ be the trajectory of the dynamical inclusion in \eqref{prec_system_simm}. We have the following:
\begin{itemize}
\item [i)] The trajectory $(x, y)$ is bounded. As a consequence, the ergodic trajectory $(\hat{x},\hat{y})$ is also bounded, and both $(x, y)$ and $(\hat{x},\hat{y})$ admit weak subsequential limit points. 
\item [ii)] For every $(u,v)\in\mathcal{X}\times\mathcal{Y}$ and for every $t\in \J$,
$$
\hat{\Delta}_{u,v}(t) \ \leq  \ \left(\int_0^t \frac{1}{\beta(s)} ds\right)^{-1} \mathcal E^S_{u,v}(0).
$$
\item [iii)] For every $(\bar{x},\bar{y})\in\mathcal{S}$ and for every $t\in \J$,
$$
\int_{0}^t \frac{d_{\bar{x}, \bar{y}}(s)}{\beta(s)}  \ ds \leq \mathcal E^S_{u,v}(0).
$$
\item [iv)] If $1/\beta \notin L^1(0,+\infty)$, then every weak  weak subsequential limit point of the ergodic trajectory $(\hat{x},\hat{y})$ is a primal-dual solution.
\item [v)] If, moreover, either $\frac{\alpha}{\beta}$ and $\frac{\delta}{\beta}$ are constant, or  there is a unique primal-dual solution, the ergodic trajectory $(\hat{x},\hat{y})$ converges weakly to a primal-dual solution.
 \end{itemize}
\end{theorem}

\begin{proof}
Young's Inequality applied to \eqref{E:parameters_symmetric} implies that, for every $(u, v)\in \mathcal{X}\times\mathcal{Y}$ and for every $t\in \J$, we have
\begin{equation}\label{E:nonneg}
    0\leq \frac{c-\norm{A}{}}{2}\left(\norm{x(t)-u}{}^2+\norm{y(t)-v}{}^2\right) \leq \mathcal E^S_{u, v}(t).
\end{equation}
Now use Lemma~\ref{L:energy} with $Z=0$, $\eta = 1/\beta$ and $\gamma=\beta$, to obtain
\begin{equation*}
    \dot V_{u,v} +\frac{1}{\beta} \Delta_{u,v} + \frac{1}{\beta} d_{u,v}= \frac{\dot\sigma}{2}\|x-u\|^2+\frac{\dot\tau}{2}\|y-v\|^2 - \big[\scal{A\dot{x}}{y-v} + \scal{A(x-u)}{\dot{y}}\big].
\end{equation*}
Since $\sigma=\frac{\alpha}{\beta}$ and $\tau=\frac{\delta}{\beta}$ are nonincreasing, it follows that 
\begin{equation} \label{E:main_inequality_symmetric} \dot{\mathcal E}^S_{u,v} + \frac{1}{\beta} \Delta_{u,v} +\frac{1}{\beta} d_{u,v} = \frac{\dot\sigma}{2}\|x-u\|^2+\frac{\dot\tau}{2}\|y-v\|^2 \leq 0.
\end{equation}   
If $(\bar{x},\bar{y})\in\mathcal{S}$, then $\Delta_{\bar{x}, \bar{y}}\geq 0$, and $d_{\bar{x}, \bar{y}} \geq 0$, whence $\dot{\mathcal E}^S_{\bar{x}, \bar{y}} \leq 0$. The quantity $\mathcal E^S_{\bar{x}, \bar{y}}(t)$ is non-increasing and, by \eqref{E:nonneg}, non-negative. Therefore, it admits a limit as $t\to +\infty$. From \eqref{E:nonneg}, it follows that $\norm{x(t)-\bar{x}}{}$ and $\norm{y(t)-\bar{y}}{}$ are bounded, which is i).

Now let $(u,v)\in\mathcal{X}\times\mathcal{Y}$.  Integrating \eqref{E:main_inequality_symmetric} on $[0,t]$, and using \eqref{E:nonneg}, we obtain
\begin{equation}\label{diseq}
\int_{0}^t\frac{1}{\beta(s)} \ \Delta_{u,v}(s) \ ds \le \mathcal E^S_{u,v}(t) + \int_{0}^t\frac{1}{\beta(s)} \ \Delta_{u,v}(s) \ ds + \int_{0}^t \frac{1}{\beta(s)} \ d_{u,v}(s) \ ds  \leq  \mathcal E^S_{u,v}(0).
\end{equation}
The convexity of the function $L(x,v)-L(u,y)$ with respect to $(x,y)$, for each fixed $(u,v)\in\mathcal{X}\times \mathcal{Y}$, together with Jensen's Inequality and \eqref{diseq}, give ii). On the other hand, setting $(u,v)=(\bar{x}, \bar{y})\in \mathcal{S}$ in \eqref{diseq}, we immediately obtain iii).

Next, if $1/\beta \notin L^1(0,+\infty)$, then $\limsup_{t\to+\infty}\hat{\Delta}_{u,v}(t)\le 0$. Pretty much in the same way as in the proof of Lemma \ref{L:limitpoints}, we show that every weak subsequential limit point of the ergodic trajectory  $(\hat{x},\hat{y})$ must belong to $\cal S$.
Finally, if $\frac{\alpha}{\beta}\equiv C_1$ and $\frac{\delta}{\beta}\equiv C_2$, system \eqref{prec_system_simm} can be rewritten as 
\begin{equation} \label{prec_system_symm_bis}
\beta (t)Q\dot{z}(t)+Mz(t) \ni 0,    
\end{equation} 
where $z=(x,y)^T$, $M$ is defined as in \eqref{eq:discr_pre}, and 
\begin{equation*}
Q=\begin{pmatrix}
    C_1 I & A^*\\ A & C_2 I,
\end{pmatrix}
\end{equation*}
is a constant symmetric preconditioning matrix, which is positive definite, in view of \eqref{eq:negder}. Reparameterize $x$, $y$ and $z$ as 
$$X(t)=x\left(\int_0^t\beta(s)\right)ds,\qquad Y(t)=x\left(\int_0^t\beta(s)\right)ds\qqbox{and}Z=(Z,Y)^T.$$
It is easy to see that $z$ is a solution of \eqref{prec_system_symm_bis} if, and only if, $Z$ is a solution of 
$$\dot Z(t)+Q^{-1}MZ(t)\ni 0.$$
The operator $Q^{-1}M$ is maximally monotone with respect to the inner product $\langle u,v\rangle_Q:=\langle Qu,v\rangle$. As a consequence, the (uniform) average of $Z$ converges weakly to a point in $\zer(Q^{-1}M)=\cal S$ (see, for instance, see \cite[Theorem 5.3]{PeySor10}). Translating this information back to $z$, we see that the ergodic trajectory $(\hat{x},\hat{y})$ converges weakly to a primal-dual solution of \eqref{opt_probl}. The other case is straightforward.
\end{proof}

\paragraph{Acknowledgments}

V. A. acknowledges financial
support by project $MIS 5154714$ of the National Recovery and Resilience Plan Greece 2.0 funded by the European Union under the NextGenerationEU Program. S. V. acknowledges the support of the European Commission (grant TraDE-OPT 861137 and grant SLING 819789), of the Ministry of Education, University and Research (grant BAC FAIR PE00000013 funded by the EU - NGEU) of the European Research Council (grant SLING 819789). S. V. and C. M. acknowledge the financial support of the US Air Force Office of Scientific Research (FA8655-22-1-7034). The research by S. V. and C. M. has been supported by the MUR Excellence Department Project awarded to Dipartimento di Matematica, Universita di Genova, CUP D33C23001110001 and MIUR (PRIN 202244A7YL).  C. M. and S. V. are members of the Gruppo Nazionale per l’Analisi Matematica, la Probabilità e le loro Applicazioni (GNAMPA) of the Istituto Nazionale di Alta Matematica (INdAM). This work represents only the view of the authors. The European Commission and the other organizations are not responsible for any use that may be made of the information it contains. J.P. acknowledges that this work benefited from the support of the FMJH Program Gaspard Monge for optimization and operations research and their interactions with data science.

\bibliographystyle{elsarticle-num-names} 
\bibliography{bibliography.bib}

\begin{thebibliography}{32}
\expandafter\ifx\csname natexlab\endcsname\relax\def\natexlab#1{#1}\fi
\providecommand{\url}[1]{\texttt{#1}}
\providecommand{\href}[2]{#2}
\providecommand{\path}[1]{#1}
\providecommand{\DOIprefix}{doi:}
\providecommand{\ArXivprefix}{arXiv:}
\providecommand{\URLprefix}{URL: }
\providecommand{\Pubmedprefix}{pmid:}
\providecommand{\doi}[1]{\href{http://dx.doi.org/#1}{\path{#1}}}
\providecommand{\Pubmed}[1]{\href{pmid:#1}{\path{#1}}}
\providecommand{\bibinfo}[2]{#2}
\ifx\xfnm\relax \def\xfnm[#1]{\unskip,\space#1}\fi
%Type = Article
\bibitem[{Chambolle and Pock(2011)}]{ChaPoc11}
\bibinfo{author}{A.~Chambolle}, \bibinfo{author}{T.~Pock},
\newblock \bibinfo{title}{A first-order primal-dual algorithm for convex
  problems with applications to imaging},
\newblock \bibinfo{journal}{Journal of Mathematical Imaging and Vision}
  \bibinfo{volume}{40} (\bibinfo{year}{2011}) \bibinfo{pages}{120--145}.
%Type = Book
\bibitem[{Attouch et~al.(2006)Attouch, Buttazzo, and Michaille}]{AttButMic06}
\bibinfo{author}{H.~Attouch}, \bibinfo{author}{G.~Buttazzo},
  \bibinfo{author}{G.~Michaille}, \bibinfo{title}{Variational analysis in
  Sobolev and BV spaces. Applications to PDE’s and Optimization},
  \bibinfo{publisher}{Society for Industrial and Applied Mathematics (SIAM)},
  \bibinfo{address}{Philadelphia}, \bibinfo{year}{2006}.
%Type = Incollection
\bibitem[{Bach et~al.(2011)Bach, Jenatton, Mairal, and Obozinski}]{BacJenMai11}
\bibinfo{author}{F.~Bach}, \bibinfo{author}{R.~Jenatton},
  \bibinfo{author}{J.~Mairal}, \bibinfo{author}{G.~Obozinski},
\newblock \bibinfo{title}{Convex optimization with sparsity-inducing norms},
\newblock in: \bibinfo{booktitle}{Optimization for machine learning},
  \bibinfo{publisher}{The MIT Press}, \bibinfo{address}{Cambridge (MA)},
  \bibinfo{year}{2011}, pp. \bibinfo{pages}{19--53}.
%Type = Article
\bibitem[{Chambolle and Pock(2016)}]{ChaPoc16}
\bibinfo{author}{A.~Chambolle}, \bibinfo{author}{T.~Pock},
\newblock \bibinfo{title}{An introduction to continuous optimization for
  imaging},
\newblock \bibinfo{journal}{Acta Numerica} \bibinfo{volume}{25}
  (\bibinfo{year}{2016}) \bibinfo{pages}{161--319}.
%Type = Article
\bibitem[{O’Donoghue et~al.(2013)O’Donoghue, Stathopoulos, and
  Boyd}]{DonStaBoy11}
\bibinfo{author}{B.~O’Donoghue}, \bibinfo{author}{G.~Stathopoulos},
  \bibinfo{author}{S.~Boyd},
\newblock \bibinfo{title}{A splitting method for optimal control},
\newblock \bibinfo{journal}{{IEEE} Transactions on Control Systems technology}
  \bibinfo{volume}{21} (\bibinfo{year}{2013}) \bibinfo{pages}{2432--2442}.
%Type = Article
\bibitem[{Yi and Pavel(2019)}]{YiPav19}
\bibinfo{author}{P.~Yi}, \bibinfo{author}{L.~Pavel},
\newblock \bibinfo{title}{An operator splitting approach for distributed
  generalized {Nash} equilibria computation},
\newblock \bibinfo{journal}{Automatica} \bibinfo{volume}{102}
  (\bibinfo{year}{2019}) \bibinfo{pages}{111--121}.
%Type = Article
\bibitem[{Vu(2013)}]{Vu13}
\bibinfo{author}{B.~C. Vu},
\newblock \bibinfo{title}{A splitting algorithm for dual monotone inclusions
  involving cocoercive operators},
\newblock \bibinfo{journal}{Advances in Computational Mathematics}
  \bibinfo{volume}{38} (\bibinfo{year}{2013}) \bibinfo{pages}{667--681}.
%Type = Article
\bibitem[{Condat(2013)}]{Con13}
\bibinfo{author}{L.~Condat},
\newblock \bibinfo{title}{A primal–dual splitting method for convex
  optimization involving {Lipschitzian}, proximable and linear composite
  terms},
\newblock \bibinfo{journal}{Journal of Optimization Theory and Applications}
  \bibinfo{volume}{158} (\bibinfo{year}{2013}) \bibinfo{pages}{460--479}.
%Type = Book
\bibitem[{Arrow et~al.(1958)Arrow, Hurwicz, and Uzawa}]{duguid1960studies}
\bibinfo{author}{K.~Arrow}, \bibinfo{author}{L.~Hurwicz},
  \bibinfo{author}{H.~Uzawa}, \bibinfo{title}{Studies in linear and non-linear
  programming: Stanford Mathematical Studies In The Social Sciences.},
  \bibinfo{publisher}{Stanford University Press}, \bibinfo{year}{1958}.
%Type = Book
\bibitem[{Arrow and Hurwicz(2014)}]{arrow2014gradient}
\bibinfo{author}{K.~J. Arrow}, \bibinfo{author}{L.~Hurwicz}, \bibinfo{title}{A
  gradient method for approximating saddle points and constrained maxima},
  \bibinfo{publisher}{Springer}, \bibinfo{address}{Basel},
  \bibinfo{year}{2014}.
%Type = Book
\bibitem[{Brezis(1973)}]{brezis1973ope}
\bibinfo{author}{H.~Brezis}, \bibinfo{title}{Op\'erateurs maximaux monotones et
  semi-groupes de contractions dans les espaces de Hilbert},
  \bibinfo{publisher}{North-Holland Pub. Co.}, \bibinfo{address}{Amsterdam},
  \bibinfo{year}{1973}.
%Type = Article
\bibitem[{Peypouquet and Sorin(2010)}]{PeySor10}
\bibinfo{author}{J.~Peypouquet}, \bibinfo{author}{S.~Sorin},
\newblock \bibinfo{title}{Evolution equations for maximal monotone operators:
  asymptotic analysis in continuous and discrete time},
\newblock \bibinfo{journal}{Journal of Convex Analysis} \bibinfo{volume}{17}
  (\bibinfo{year}{2010}) \bibinfo{pages}{1113--1163}.
%Type = Inproceedings
\bibitem[{Holding and Lestas(2014)}]{holding2014convergence}
\bibinfo{author}{T.~Holding}, \bibinfo{author}{I.~Lestas},
\newblock \bibinfo{title}{On the convergence to saddle points of concave-convex
  functions, the gradient method and emergence of oscillations},
\newblock in: \bibinfo{booktitle}{53rd IEEE Conference on Decision and
  Control}, \bibinfo{organization}{IEEE}, \bibinfo{year}{2014}, pp.
  \bibinfo{pages}{1143--1148}.
%Type = Article
\bibitem[{Cherukuri et~al.(2017{\natexlab{a}})Cherukuri, Gharesifard, and
  Cortes}]{cherukuri2017saddle}
\bibinfo{author}{A.~Cherukuri}, \bibinfo{author}{B.~Gharesifard},
  \bibinfo{author}{J.~Cortes},
\newblock \bibinfo{title}{Saddle-point dynamics: conditions for asymptotic
  stability of saddle points},
\newblock \bibinfo{journal}{SIAM Journal on Control and Optimization}
  \bibinfo{volume}{55} (\bibinfo{year}{2017}{\natexlab{a}})
  \bibinfo{pages}{486--511}.
%Type = Article
\bibitem[{Cherukuri et~al.(2017{\natexlab{b}})Cherukuri, Mallada, Low, and
  Cort{\'e}s}]{cherukuri2017role}
\bibinfo{author}{A.~Cherukuri}, \bibinfo{author}{E.~Mallada},
  \bibinfo{author}{S.~Low}, \bibinfo{author}{J.~Cort{\'e}s},
\newblock \bibinfo{title}{The role of convexity in saddle-point dynamics:
  Lyapunov function and robustness},
\newblock \bibinfo{journal}{IEEE Transactions on Automatic Control}
  \bibinfo{volume}{63} (\bibinfo{year}{2017}{\natexlab{b}})
  \bibinfo{pages}{2449--2464}.
%Type = Article
\bibitem[{Holding and Lestas(2020{\natexlab{a}})}]{holding2020stability}
\bibinfo{author}{T.~Holding}, \bibinfo{author}{I.~Lestas},
\newblock \bibinfo{title}{Stability and instability in saddle point dynamics -
  part i},
\newblock \bibinfo{journal}{IEEE Transactions on Automatic Control}
  \bibinfo{volume}{66} (\bibinfo{year}{2020}{\natexlab{a}})
  \bibinfo{pages}{2933--2944}.
%Type = Article
\bibitem[{Holding and Lestas(2020{\natexlab{b}})}]{holding2020stabilityt}
\bibinfo{author}{T.~Holding}, \bibinfo{author}{I.~Lestas},
\newblock \bibinfo{title}{Stability and instability in saddle point dynamics -
  part ii: The subgradient method},
\newblock \bibinfo{journal}{IEEE Transactions on Automatic Control}
  \bibinfo{volume}{66} (\bibinfo{year}{2020}{\natexlab{b}})
  \bibinfo{pages}{2945--2960}.
%Type = Article
\bibitem[{Cherukuri et~al.(2016)Cherukuri, Mallada, and
  Cort{\'e}s}]{cherukuri2016asymptotic}
\bibinfo{author}{A.~Cherukuri}, \bibinfo{author}{E.~Mallada},
  \bibinfo{author}{J.~Cort{\'e}s},
\newblock \bibinfo{title}{Asymptotic convergence of constrained primal--dual
  dynamics},
\newblock \bibinfo{journal}{Systems \& Control Letters} \bibinfo{volume}{87}
  (\bibinfo{year}{2016}) \bibinfo{pages}{10--15}.
%Type = Article
\bibitem[{Feijer and Paganini(2010)}]{feijer2010stability}
\bibinfo{author}{D.~Feijer}, \bibinfo{author}{F.~Paganini},
\newblock \bibinfo{title}{Stability of primal--dual gradient dynamics and
  applications to network optimization},
\newblock \bibinfo{journal}{Automatica} \bibinfo{volume}{46}
  (\bibinfo{year}{2010}) \bibinfo{pages}{1974--1981}.
%Type = Article
\bibitem[{Niederl{\"a}nder(2024)}]{niederlander2024arrow}
\bibinfo{author}{S.~K. Niederl{\"a}nder},
\newblock \bibinfo{title}{{On the Arrow--Hurwicz differential system for
  linearly constrained convex minimization}},
\newblock \bibinfo{journal}{Optimization} \bibinfo{volume}{73}
  (\bibinfo{year}{2024}) \bibinfo{pages}{2313--2345}.
%Type = Article
\bibitem[{Battahi et~al.(2024)Battahi, Chbani, Niederl{\"a}nder, and
  Riahi}]{battahi2024asymptotic}
\bibinfo{author}{F.~Battahi}, \bibinfo{author}{Z.~Chbani},
  \bibinfo{author}{S.~Niederl{\"a}nder}, \bibinfo{author}{H.~Riahi},
\newblock \bibinfo{title}{{Asymptotic behavior of the {A}rrow-{H}urwicz
  differential system with Tikhonov regularization}},
\newblock \bibinfo{journal}{arXiv preprint arXiv:2411.17656}
  (\bibinfo{year}{2024}).
%Type = Article
\bibitem[{Ozaslan et~al.(2024)Ozaslan, Patrinos, and
  Jovanovi{\'c}}]{ozaslan2024stability}
\bibinfo{author}{I.~Ozaslan}, \bibinfo{author}{P.~Patrinos},
  \bibinfo{author}{M.~Jovanovi{\'c}},
\newblock \bibinfo{title}{Stability of primal-dual gradient flow dynamics for
  multi-block convex optimization problems},
\newblock \bibinfo{journal}{arXiv preprint arXiv:2408.15969}
  (\bibinfo{year}{2024}).
%Type = Article
\bibitem[{Li and Shi(2024)}]{LiShi24}
\bibinfo{author}{B.~Li}, \bibinfo{author}{B.~Shi},
\newblock \bibinfo{title}{Understanding the {ADMM} algorithm via
  high-resolution differential equations},
\newblock \bibinfo{journal}{arXiv preprint arXiv:2401.07096}
  (\bibinfo{year}{2024}).
%Type = Inproceedings
\bibitem[{Apidopoulos et~al.(2023)Apidopoulos, Molinari, Rosasco, and
  Villa}]{apidopoulos2023regularization}
\bibinfo{author}{V.~Apidopoulos}, \bibinfo{author}{C.~Molinari},
  \bibinfo{author}{L.~Rosasco}, \bibinfo{author}{S.~Villa},
\newblock \bibinfo{title}{Regularization properties of dual subgradient flow},
\newblock in: \bibinfo{booktitle}{2023 European Control Conference (ECC)},
  \bibinfo{organization}{IEEE}, \bibinfo{year}{2023}, pp.
  \bibinfo{pages}{1--8}.
%Type = Article
\bibitem[{Attouch et~al.(2022)Attouch, Chbani, Fadili, and
  Riahi}]{attouch2022fast}
\bibinfo{author}{H.~Attouch}, \bibinfo{author}{Z.~Chbani},
  \bibinfo{author}{J.~Fadili}, \bibinfo{author}{H.~Riahi},
\newblock \bibinfo{title}{Fast convergence of dynamical {ADMM} via time scaling
  of damped inertial dynamics},
\newblock \bibinfo{journal}{Journal of Optimization Theory and Applications}
  \bibinfo{volume}{193} (\bibinfo{year}{2022}) \bibinfo{pages}{704--736}.
%Type = Article
\bibitem[{Sun et~al.(2024)Sun, He, and Long}]{sun2024inertial}
\bibinfo{author}{X.~Sun}, \bibinfo{author}{L.~He}, \bibinfo{author}{X.~Long},
\newblock \bibinfo{title}{Inertial primal-dual dynamics with {H}essian-driven
  damping and {T}ikhonov regularization for convex-concave bilinear saddle
  point problems},
\newblock \bibinfo{journal}{arXiv preprint arXiv:2412.05931}
  (\bibinfo{year}{2024}).
%Type = Article
\bibitem[{Ding et~al.(2024)Ding, Fliege, and Vuong}]{ding2024fast}
\bibinfo{author}{K.~Ding}, \bibinfo{author}{J.~Fliege}, \bibinfo{author}{P.~T.
  Vuong},
\newblock \bibinfo{title}{Fast convergence of the primal-dual dynamical system
  and corresponding algorithms for a nonsmooth bilinearly coupled saddle point
  problem},
\newblock \bibinfo{journal}{Computational Optimization and Applications}
  (\bibinfo{year}{2024}) \bibinfo{pages}{1--42}.
%Type = Article
\bibitem[{Battahi et~al.(2025)Battahi, Chbani, and
  Riahi}]{battahi2025simultaneous}
\bibinfo{author}{F.~F. Battahi}, \bibinfo{author}{Z.~Chbani},
  \bibinfo{author}{H.~Riahi},
\newblock \bibinfo{title}{On the simultaneous convergence of values and
  trajectories of continuous inertial dynamics with {T}ikhonov regularization
  to solve convex minimization with affine constraints.},
\newblock \bibinfo{journal}{Applied Set-Valued Analysis \& Optimization}
  \bibinfo{volume}{7} (\bibinfo{year}{2025}).
%Type = Article
\bibitem[{Luo(2022)}]{luo2022primal}
\bibinfo{author}{H.~Luo},
\newblock \bibinfo{title}{A primal-dual flow for affine constrained convex
  optimization},
\newblock \bibinfo{journal}{ESAIM: Control, Optimisation and Calculus of
  Variations} \bibinfo{volume}{28} (\bibinfo{year}{2022}) \bibinfo{pages}{33}.
%Type = Article
\bibitem[{Li and Shi(2024)}]{li2024understanding}
\bibinfo{author}{B.~Li}, \bibinfo{author}{B.~Shi},
\newblock \bibinfo{title}{Understanding the {PDHG} algorithm via
  high-resolution differential equations},
\newblock \bibinfo{journal}{arXiv preprint arXiv:2403.11139}
  (\bibinfo{year}{2024}).
%Type = Article
\bibitem[{Molinari et~al.(2024)Molinari, Massias, Rosasco, and Villa}]{itreg}
\bibinfo{author}{C.~Molinari}, \bibinfo{author}{M.~Massias},
  \bibinfo{author}{L.~Rosasco}, \bibinfo{author}{S.~Villa},
\newblock \bibinfo{title}{Iterative regularization for low complexity
  regularizers},
\newblock \bibinfo{journal}{Numer. Math.} \bibinfo{volume}{156}
  (\bibinfo{year}{2024}) \bibinfo{pages}{641--689}.
%Type = Article
\bibitem[{Combettes and Wajs(2005)}]{combettes2005signal}
\bibinfo{author}{P.~L. Combettes}, \bibinfo{author}{V.~R. Wajs},
\newblock \bibinfo{title}{Signal recovery by proximal forward-backward
  splitting},
\newblock \bibinfo{journal}{Multiscale modeling \& simulation}
  \bibinfo{volume}{4} (\bibinfo{year}{2005}) \bibinfo{pages}{1168--1200}.

\end{thebibliography}

\end{document}